\documentclass[11pt]{amsart}
\usepackage{amscd,amssymb,subfigure}
\usepackage[all]{xy}
\usepackage{graphicx}
\usepackage[colorlinks,plainpages,backref,urlcolor=blue]{hyperref}

\newtheorem{theorem}{Theorem}[section]
\newtheorem{corollary}[theorem]{Corollary}
\newtheorem{lemma}[theorem]{Lemma}
\newtheorem{prop}[theorem]{Proposition}

\newtheorem{thm}{Theorem}

\theoremstyle{definition}
\newtheorem{definition}[theorem]{Definition}
\newtheorem{example}[theorem]{Example}
\newtheorem{remark}[theorem]{Remark}
\newtheorem*{ack}{Acknowledgments}

\newcommand{\Z}{\mathbb{Z}}
\newcommand{\Q}{\mathbb{Q}}
\newcommand{\R}{\mathbb{R}}
\newcommand{\C}{\mathbb{C}}
\newcommand{\F}{\mathbb{F}}

\renewcommand{\AA}{\mathbb{A}}

\renewcommand{\k}{\Bbbk}

\newcommand{\A}{{\mathcal{A}}}

\newcommand{\As}{A}
\newcommand{\Ps}{\mathsf{P}}
\newcommand{\Qs}{\mathsf{Q}}
\newcommand{\Rs}{\mathsf{R}}
\newcommand{\Ss}{\mathsf{S}}

\DeclareMathOperator{\rank}{rank}
\DeclareMathOperator{\gr}{gr}
\DeclareMathOperator{\im}{im}

\DeclareMathOperator{\id}{id}
\DeclareMathOperator{\ab}{{ab}}
\DeclareMathOperator{\Sym}{Sym}
\DeclareMathOperator{\ch}{char}
\DeclareMathOperator{\GL}{GL}
\DeclareMathOperator{\Hom}{{Hom}}

\DeclareMathOperator{\proj}{pr}
\DeclareMathOperator{\ev}{ev}
\DeclareMathOperator{\linn}{\,lin}
\DeclareMathOperator{\Mat}{{Mat}}

\newcommand{\wX}{\widetilde{X}}
\newcommand{\wY}{\widetilde{Y}}

\newcommand{\same}{\Longleftrightarrow}
\newcommand{\surj}{\twoheadrightarrow}
\newcommand{\inj}{\hookrightarrow}

\newcommand{\DS}{\displaystyle}
\newcommand{\brac}[1]{[\![ #1 ]\!]}
\newcommand{\tr}{{\scriptstyle{\mathsf{T}}}}
\newcommand{\dnu}{\partial^{\nu}}

\begin{document}

\title[Spectral sequence of an equivariant chain complex]{%
The spectral sequence of an equivariant chain complex and 
homology with local coefficients}

\author[Stefan Papadima]{Stefan Papadima$^1$}
\address{Institute of Mathematics Simion Stoilow, 
P.O. Box 1-764,
RO-014700 Bucharest, Romania}
\email{Stefan.Papadima@imar.ro}
\thanks{$^1$Partially supported by the CEEX Programme of the 
Romanian Ministry of Education and Research, contract
2-CEx 06-11-20/2006}

\author[Alexander~I.~Suciu]{Alexander~I.~Suciu$^2$}
\address{Department of Mathematics,
Northeastern University,
Boston, MA 02115, USA}
\email{a.suciu@neu.edu}
\urladdr{http://www.math.neu.edu/\~{}suciu}
\thanks{$^2$Partially supported by NSF grant DMS-0311142}



\begin{abstract}
We study the spectral sequence associated to the filtration 
by powers of the augmentation ideal on the (twisted) equivariant 
chain complex of the universal cover of a connected 
CW-complex $X$.  In the process, we identify the $d^1$ 
differential in terms of the coalgebra structure of $H_*(X,\k)$, 
and the $\k\pi_1(X)$-module structure on the twisting coefficients. 
In particular, this recovers in dual form a result of Reznikov, 
on the mod $p$ cohomology of cyclic $p$-covers of
aspherical complexes. This approach provides 
information on the homology of all Galois covers of $X$. 
It also yields computable upper bounds on the ranks 
of the cohomology groups of $X$, with coefficients in 
a prime-power order, rank one local system.  When 
$X$ admits a minimal cell decomposition, we relate 
the linearization of the equivariant cochain complex 
of the universal abelian cover to the Aomoto complex, 
arising from the cup-product structure of $H^*(X,\k)$, 
thereby generalizing a result of Cohen and Orlik. 
\end{abstract}
\maketitle

\tableofcontents

\section{Introduction}
\label{sect:intro}

\subsection{The equivariant chain complex} 
\label{intro:eq chain}
In his pioneering work from the late 1940s, J.H.C.~Whitehead 
established the category of CW-complexes as the natural 
framework for much of homotopy theory. In \cite{Wh}, 
he highlighted the key role played by the cellular chain 
complex of the universal cover, $\wX$, of a connected 
CW-complex $X$. Among other things, Whitehead showed 
that a map $f\colon X\to Y$ is a homotopy equivalence if 
and only if the induced map between equivariant chain complexes, 
$\tilde{f_*}\colon C_{\bullet} (\wX,\Z)\to C_{\bullet} (\wY,\Z)$, is 
an equivariant chain-homotopy equivalence.  For $2$-dimensional 
complexes, the fundamental group, together with the equivariant 
chain-homotopy type of $C_{\bullet} (\wX,\Z)$, constitute a complete 
set of homotopy type invariants: if $X$ and $Y$ are two such spaces, 
with $\pi_1(X) \cong \pi_1(Y)$ and $C_{\bullet} (\wX,\Z)\simeq 
C_{\bullet} (\wY,\Z)$, then $X\simeq Y$. 

As noted by S.~Eilenberg \cite{Ei}, the equivariant chain complex 
is tightly connected to homology with twisted coefficients: 
Given a linear representation, $\rho\colon \pi_1(X) \to \GL(n,\k)$, 
the homology of $X$ with coefficients in the local system 
${}_{\rho}\k^n$ is  $H_* ({}_{\rho}\k^n \otimes_{\k \pi_1(X)} 
C_{\bullet}(\wX,\k))$. 

In this paper, we revisit these classical topics, drawing much 
of the motivation from recent work on the topology of complements 
of hyperplane arrangements, and the study of cohomology 
jumping loci. One of our main goals is to give tight upper 
bounds for the twisted Betti ranks, computable in terms of 
much simpler data, involving only ordinary cohomology. 
The basic tool for our approach is a spectral 
sequence, which we now proceed to describe. 

\subsection{The equivariant spectral sequence}
\label{intro:thmA}
Let $\k\pi$ be the group ring of a group $\pi$, over 
a coefficient ring $\k$, and let $M$ be a right $\k\pi$-module.  
The successive powers of the augmentation ideal, $I=I_{\k}(\pi)$, 
determine a filtration on $M$; the associated graded object, 
$\gr(M)=\bigoplus_{n\ge 0} MI^n /MI^{n+1}$, is a module over the 
ring $\gr(\k\pi)$. 

Now let $X$ be a connected CW-complex, with 
fundamental group $\pi=\pi_1(X)$, and let 
$C_{\bullet}(X, M)=M\otimes_{\k\pi} C_{\bullet}(\wX,\k)$ 
be the equivariant chain complex of $X$ with coefficients in $M$, 
discussed in \S\ref{sec:equivchain}.
The $I$-adic filtration on $C_{\bullet}(X, M)$ is compatible 
with the boundary maps, and thus gives rise to a spectral sequence, 
$E^{\bullet}(X, M)$, as explained in \S\ref{sec:spectral}. 
Under some mild conditions, we identify 
in \S\ref{sec:e1 page} 
the differential $d^1\colon E^1\to E^1$, solely in terms of the 
coalgebra structure of $H_*(X,\k)$ and the $\gr(\k\pi)$-module 
structure on $\gr(M)$.  

\begin{thm}
\label{thm:A}
There is a second-quadrant spectral sequence,  
$\{E^r(X,M),d^r\}_{r\ge 1}$, with 
$E^1_{-p,p+q}(X,M)=H_{q}(X,\gr^p(M))$. 
If $\k$ is a field, or $\k=\Z$ and $H_*(X,\Z)$ is torsion-free, 
then $E^1_{-p,p+q}(X,M)=\gr^p(M) \otimes_{\k} H_{q}(X,\k)$, and 
the $d^1$ differential decomposes as 
\[
\xymatrixcolsep{42pt}
\xymatrixrowsep{10pt}
\xymatrix{ \gr^p(M) \otimes_{\k} H_q \ar^(.4){\id \otimes \nabla_X}[r] 
& \gr^p(M)\otimes_{\k} (H_1\otimes_{\k} H_{q-1}) \ar^{\cong}[d]
\\
& (\gr^p(M)\otimes_{\k} \gr^1(\k\pi))\otimes_{\k} H_{q-1}  
\ar^(.57){\gr(\mu_M) \otimes \id}[r]
& \gr^{p+1}(M) \otimes_{\k}  H_{q-1}\,,
}
\]
where $\nabla_X$ is the comultiplication map on $H_*=H_*(X,\k)$, 
and $\mu_M\colon M\otimes_{\k} \k\pi\to M$ is the multiplication map. 
\end{thm}

Full details are given in Theorem~\ref{thm:d1map}. 
In the case when $X$ is of finite type, $d^1$ is determined by the 
cup-product structure in $H^*(X,\k)$ and the map $\gr(\mu_M)$.  

The idea to use powers of augmentation ideals 
to define a second quadrant spectral sequence 
in terms of group presentations  goes back to J.~Stallings \cite{St}. 
For more on the Stallings spectral sequence, see \cite{Gr}, \cite{Wa}. 

In \S\ref{sec:convergence}, we study the convergence 
properties of the spectral sequence from Theorem \ref{thm:A}. 
Under fairly general assumptions, $E^{\bullet}(X,M)$ has an 
$E^{\infty}$ term. Yet, as we show in Example \ref{ex:non conv}, 
there are finite CW-complexes $X$ for which $E^{\bullet}(X,\k\pi)$ 
does not converge. 

\subsection{Base change}
\label{intro:base change}
To obtain more structure in the spectral 
sequence, we restrict in \S\ref{sect:change rings} to a special 
situation. Suppose $\nu\colon \pi\surj G$ is an  
epimorphism onto a group $G$; then the group ring $\k{G}$ 
becomes a right $\k\pi$-module, via extension of scalars.  
The resulting spectral sequence, $E^{\bullet}(X,\k{G}_{\nu})$, 
is a spectral sequence in the category of left $\gr_J(\k{G})$-modules, 
where $J$ is the augmentation ideal of $\k{G}$. In 
Proposition \ref{prop:d1 kg}, we describe the differential 
$d^1_G$, solely in terms of the  induced homomorphism, 
$\nu_*\colon H_1(X,\k) \to H_1(G,\k)$, and the comultiplication 
map, $\nabla_X$.

Note that $H_*(X,\k{G}_{\nu})=H_*(Y,\k)$, where $Y\to X$ is 
the Galois $G$-cover defined by $\nu$. The homology groups 
of $Y$ support two natural filtrations:  $F^{\bullet}$, 
coming from the spectral sequence, and $J^{\bullet}$, 
by the powers of the augmentation ideal. We then have 
an inclusion, $J^k \cdot  H_*(Y,\k) \subseteq F^{k} H_*(Y,\k)$.
Equality holds for $G=\Z$, as we show in Lemma \ref{lem:equal filt},
but in general the two filtrations do not agree, even when $G=\Z^2$.

In \S\ref{sec:completions} we study the more general 
situation when $\nu\colon \pi\surj G$ is an epimorphism 
to an abelian group, making use of the general machinery 
developed by  J.P.~Serre in \cite{Ser}. Assuming $X$ is of 
finite type and $\k$ is a field, the spectral sequence 
$E^{\bullet}(X,\k{G}_{\nu})$ converges, and computes 
the $J$-adic completion of $H_*(X,\k{G}_{\nu})$. Moreover, 
if $X$ is a finite CW-complex, the spectral sequence 
collapses in finitely many steps.  

In the case when $X$ is a $K(\pi,1)$ space, $G=\Z_p$, 
and $\k=\F_p$, discussed separately in \S\ref{sec:kzp}, 
the spectral sequence $E^{\bullet}(X,\k{G}_{\nu})$ is the 
homological version of a spectral sequence first considered 
by A.~Reznikov in \cite{Re}.

\subsection{Monodromy action}
\label{intro:thmB}

In \S\ref{sec:mono}, we focus exclusively on 
infinite cyclic covers, analyzing the homology 
groups $H_{q}(X,\k\Z_{\nu})$, viewed as modules over 
the Laurent polynomial ring, $\k\Z=\k[t^{\pm 1}]$. 
We assume $\k$ is a field, so that  $\k[t^{\pm 1}]$ 
is a PID.    Given an element $a\in H^1(X,\k)$ 
with $a^2=0$, left-multiplication by $a$ turns the 
cohomology ring of $X$ into a cochain complex, 
$(H^*( X,\k),\cdot a)\colon H^0(X,\k) \xrightarrow{a}  
H^1(X,\k) \xrightarrow{a}  H^2(X,\k) \xrightarrow{a} \cdots$. 

\begin{thm}
\label{thm:B}
Let $X$ be a connected, finite-type CW-complex, 
$\nu\colon \pi_1(X)\surj \Z$ an epimorphism, 
and $\nu_{\k}\in H^1(X,\k)$ the corresponding 
cohomology class. Then, for all $q\ge 0$:
\begin{enumerate}
\item \label{B1}
The $\gr_J(\k\Z)$-module structure on $E^{\infty}(X, \k\Z_{\nu})$ 
determines  $P^q_0$ and $P^{q}_{t-1}$, the free and 
$(t-1)$-primary parts of $H_q(X, \k\Z_{\nu})$, viewed 
as $\k[t^{\pm 1}]$-modules. 
\item \label{B2}
The monodromy action of $\Z$ on $P^{j}_0 \oplus 
P^{j}_{t-1}$ is trivial for all $j\le q$ if and only if 
$H^j(H^*( X,\k) , \cdot\nu_{\k})=0$, for all $j\le q$. 
\end{enumerate}
\end{thm}

For a more detailed statement, see Propositions \ref{prop:monospsq} 
and \ref{prop:trivial action}. Part \eqref{B2} has an analog, for
$\nu\colon \pi_1(X) \surj \Z_p$ and $\k= \F_p$ ($p\ne 2$); see
Proposition \ref{prop:jordan block}.

Particularly interesting is the case of a smooth manifold $X$ 
fibering over the circle, with $\nu\colon \pi \surj \Z$ the 
homomorphism induced on $\pi_1$ by the projection map, 
$p\colon X\to S^1$.  
The homology of the resulting infinite cyclic cover was studied 
by J.~Milnor in \cite{Mil2}.  This led to another spectral sequence, 
introduced by M.~Farber in \cite{Fa85}, and further developed by  
S.P.~Novikov in \cite{Nov}.  The Farber-Novikov spectral sequence 
has $(E_1, d_1)$-page dual to our $(E^1(X,\k\Z_{\nu}), d^1_{\nu})$-page, 
and higher differentials given by certain Massey products, 
see \cite{Fa}. Their spectral sequence, though, converges 
to the free part of $H_*(X,\k\Z_\nu)$, and thus misses 
the information on the $(t-1)$-primary part captured by the 
equivariant spectral sequence.  The spectral sequence 
$E^{\bullet}(X,\k\Z_{\nu})$ was also investigated by G.~Denham 
in \cite{De}, in the special case when $X$ is the complement 
of a complexified arrangement of real hyperplanes through 
the origin of $\R^{\ell}$, and $p$ is the Milnor fibration. 

Theorem \ref{thm:B} has already found several 
applications in the literature. In \cite{PST}, we analyze the 
monodromy action on the homology of Galois $\Z$--covers, for
toric complexes associated to finite simplicial complexes.
Using Part \eqref{B2} of the Theorem, we obtain a combinatorial 
criterion for the full triviality of this action, up to a given degree.
In \cite{PS08}, Theorem \ref{thm:B}\eqref{B2} yields a new formality 
criterion, and a purely topological proof of a basic result on the 
monodromy action, for the homology of Milnor fibers of plane 
curve singularities.

\subsection{Bounds on twisted Betti ranks}
\label{intro:thmD}
Computing  cohomology groups with coefficients in a 
rank $1$ local system can be an arduous task.  It is 
thus desirable to have efficient, readily computable 
bounds for the  ranks of these groups.  We supply 
such bounds in \S\S\ref{sect:betti bounds}--\ref{sect:aomoto bounds}. 

Given a non-zero complex number $\zeta\in \C^{\times}$ and 
a homomorphism $\nu\colon \pi\to \Z$, define a representation  
$\rho\colon \pi\to \C^{\times}$ by $\rho(g)=\zeta^{\nu(g)}$.  
If $\zeta$ is a $d$-th root of unity, such a homomorphism $\rho$ 
is called a rational character (of order $d$).

\begin{thm}
\label{thm:D}
Let $X$ be a connected, finite-type CW-complex, and let 
$\rho\colon \pi_1(X)\to \C^{\times}$ be a rational character, 
of prime-power order $d=p^r$. Then, for all $q\ge 0$, 
\[
\dim_{\C} H^q(X, {}_{\rho}\C) \le 
\dim_{\F_p} H^q(X,\F_p).
\]
If, moreover, $H_*(X,\Z)$ is torsion-free, 
\[
\dim_{\C} H^q(X, {}_{\rho}\C) \le 
\dim_{\F_p} H^q(H^{*}(X,\F_p),\nu_{\F_p}).
\]
\end{thm}

The proof is given in Theorems \ref{thm:bettibound} and  
\ref{thm:cohobound}. Neither of these two 
inequalities can be sharpened further: we give examples 
showing that both the prime-power order hypothesis on 
$d$ and the torsion-free hypothesis on $H_*(X,\Z)$ 
are really necessary.

The second inequality above generalizes a result of 
D.~Cohen and P.~Orlik (\cite[Theorem 1.3]{CO}), 
valid only when $X$ is the complement 
of a complex hyperplane arrangement, and $r=1$.
This inequality is used in a crucial way in \cite{MP}.
It yields a purely combinatorial description of the monodromy 
action on the degree $1$ rational homology of the Milnor fiber, 
for arbitrary subarrangements of classical Coxeter arrangements.

\subsection{Minimality and linearization}
\label{intro:thmE}
Using the equivariant spectral sequence of a Galois cover, 
we give in \S\ref{sect:minilin}  an intrinsic meaning to
linearization of equivariant (co)chain complexes, in the  
case when the CW-complex $X$ admits a cell structure 
with minimal number of cells in each dimension. 

Pick a basis $\{ e_1,\dots, e_n \}$ for $H_1=H_1(X,\k)$, 
and identify the symmetric algebra $S=\Sym(H_1)$ 
with the polynomial ring $\k [e_1,\dots, e_n]$. The 
{\em universal Aomoto complex} of $H^*=H^*(X,\k)$ 
is the cochain complex of free $S$-modules, 
\[
\xymatrix{H^0 \otimes_{\k} S \ar^{D^0}[r] & 
H^1 \otimes_{\k} S \ar^{D^1}[r] & 
H^2 \otimes_{\k} S \ar^(.6){D^2}[r] & 
\cdots},
\]
with differentials 
$D(\alpha \otimes 1)= \sum_{i=1}^{n} 
e_i^* \cdot \alpha \otimes e_i$. 

\begin{thm}
\label{thm:E}
Let $X$ be a minimal CW-complex, and assume $\k=\Z$, or a field. 
\begin{enumerate}
\item \label{E1}
Let $\nu\colon \pi\surj G$ be an epimorphism. Then 
the linearization of the equivariant chain complex,
$(C_{\bullet}(X,\k G_{\nu}), \tilde{\partial}_{\bullet}^G)$, is
equal to the $E^1$-term of the equivariant spectral sequence,
$(E^1(X,\k G_{\nu}), d_G^1)$. 
\item \label{E2}
The linearization of the equivariant cochain complex 
of the universal abelian cover of $X$ coincides 
with the universal Aomoto complex of $H^*(X,\k)$.
\end{enumerate}
\end{thm}

Part \eqref{E1}---suitably interpreted with the aid of 
Theorem \ref{thm:A}---was proved in \cite[Theorem 20]{DP1}, 
in the case when $\nu =\id$ and $\k= \Z$, and under the additional 
assumption that the cohomology ring $H^*(X, \Z)$ is generated 
in degree one. 

Part \eqref{E2}  generalizes \cite[Theorem 1.2]{CO}, valid 
only for complements of complex hyperplane arrangements, 
and $\k=\C$.  Earlier results in this direction, also within the 
confines of arrangement theory, were obtained 
in \cite{CS99} and \cite{C}. The result in Part \eqref{E2} 
was recently proved by M.~Yoshinaga \cite{Yo}, under 
an additional condition, satisfied by arrangement 
complements, but not by arbitrary minimal CW-complexes.  
Theorem \ref{thm:E}\eqref{E2} shows that this condition is 
unnecessary, thereby answering a question raised by 
Yoshinaga in Remark~15, preprint version of \cite{Yo}. 

\section{The equivariant chain complex of a CW-complex}
\label{sec:equivchain}

In this section, we review a well-known construction, going  
back to J.H.C. Whitehead \cite{Wh} and S. Eilenberg \cite{Ei}: 
the chain complex of the universal cover of a cell complex $X$, 
with coefficients in a $\k\pi_1(X)$-module $M$.   We start with 
some basic algebraic notions.  

\subsection{Associated graded rings}
\label{subsec:gring}

Let $R$ be a ring, and $J$ a two-sided ideal. The successive powers 
of $J$ determine a descending filtration on $R$, called the 
{\em $J$-adic filtration}.  The {\em associated graded ring}, 
$\gr_J(R)$, is defined as 
\begin{equation}
\label{eq:gr R}
\gr_J(R)=\bigoplus_{n\ge 0} J^n /J^{n+1}.  
\end{equation}
On homogeneous components, the multiplication map 
$\gr_J^n(R) \otimes \gr_J^m(R) \to \gr_J^{n+m}(R)$ 
is induced from multiplication in the ring $R$. 

We will be mainly interested in the following situation. 
Let $\pi$ be a group, and let $\k$ 
be a commutative ring with unit $1$.  The {\em group ring} 
of $\pi$, denoted  $\k\pi$, consists of all finite linear 
combinations of group elements with coefficients in $\k$.  
Multiplication in $\k\pi$ is induced from the group operation 
on $\pi$; the unit (denoted by $1$) is $1\cdot \iota$, where 
$\iota$ is the identity of $\pi$. The {\em  augmentation ideal}, 
$I=I_{\k}\pi$, is the kernel of the ring homomorphism
$\epsilon \colon \k\pi \to \k$, $\sum n_g g\mapsto \sum n_g$; 
as a $\k$-module, $I$ is freely generated by the elements  
$g-1$, for $g\ne \iota$. 

The augmentation ideal defines the $I$-adic filtration on $\k\pi$; 
let $\gr(\k\pi)=\gr_I(\k\pi)$ be the associated graded ring. 
Note that $\gr(\k\pi)$ is always generated as a ring in 
degree $1$.  Clearly, $\gr^0(\k\pi) \cong \k$.  Moreover, 
the map $\pi \to I$, $g\mapsto g-1$ factors through  
an isomorphism $H_1(\pi,\k)\to I/I^2$ (see \cite{HS}).  Thus,
\begin{equation}
\label{eq:gr1}
\gr^1(\k\pi)\cong H_1(\pi,\k).
\end{equation}

As an example, consider the free abelian group $\pi=\Z^n$. 
The group ring $\k\Z^n$ can be identified with the Laurent 
polynomial ring $\k[t_1^{\pm 1},\dots ,t_n^{\pm 1}]$, while 
$\gr(\k\Z^n)$ can be identified with $\k[x_1,\dots ,x_n]$,  
the polynomial ring in $n$ variables, via the map 
$t_i-1\mapsto x_i$.  

\subsection{Associated graded modules}
\label{subsec:func1}

A similar construction applies to modules: if $M$ is a left 
$\k\pi$-module, then $M$ is filtered by the submodules 
$\{ I^n M \}_{n\ge 0}$.  The associated graded object, 
$\gr_I(M)=\bigoplus_{n\ge 0}  I^n M / I^{n+1} M$, 
inherits a graded module structure over $\gr(\k\pi)$.  
This module is generated in degree $0$ 
by the module of coinvariants, $M/{I M}$.  

All these constructions are functorial.  For example, 
if $\alpha\colon \pi\to \pi'$ is a group homomorphism, 
then the linear extension to group rings, 
$\bar\alpha\colon \k\pi\to \k\pi'$, is a ring map, 
preserving the respective $I$-adic filtrations.  
Thus, $\bar\alpha$ induces a degree $0$ ring homomorphism, 
$\gr(\alpha)\colon \gr(\k\pi)\to \gr(\k\pi')$. 

Given a $\k\pi$-module $M$ and a $\k\pi'$-module $M'$, 
a $\k$-linear map $\phi\colon M\to M'$ is said to be equivariant 
with respect to the ring map $\bar\alpha\colon \k\pi\to \k\pi'$ if 
$\phi(gm)=\alpha(g) \phi(m)$, for all $g\in \pi$ and $m\in M$.  
Such a map $\phi$ preserves $I$-adic filtrations, and thus 
induces a map $\gr(\phi)\colon \gr_I(M)\to \gr_I(M')$, 
equivariant with respect to $\gr(\alpha)$. 

Completely analogous considerations apply to right 
$\k\pi$-modules.  

\subsection{Chains on the universal cover}
\label{subsec:equivchain}
Let $X$ be a connected CW-complex, with skeleta 
$\{X^q\}_{q\ge 0}$.  Up to homotopy, we may assume $X$ has 
a single $0$-cell, call it $e_0$, which we will take as the 
basepoint.  Moreover, we may assume all attaching 
maps $(S^q,*) \to (X^q,e_0)$ are basepoint-preserving. 

Fix a commutative ring $\k$ with unit, and denote by 
$C_{\bullet}(X, \k)=(C_q(X, \k),\partial_q)_{q\ge 0}$ 
the cellular chain complex of $X$, with coefficients in $\k$. 
Recall $C_q(X, \k)=H_q(X^q,X^{q-1},\k)$ is a free $\k$-module, 
with basis indexed by the $q$-cells of $X$. 

Let $p\colon \wX \to X$ be the universal cover.  
The cell structure on $X$ lifts to a cell structure on $\wX$.  
Fixing a lift $\tilde{e}_0\in p^{-1}(e_0)$ identifies 
the fundamental group $\pi=\pi_1(X,e_0)$ with the 
group of deck transformations of $\wX$, which 
permute the cells.  Therefore, we may view 
\begin{equation}
\label{eq:equiv}
C_{\bullet}(\wX, \k)=(C_q(\wX, \k),\tilde{\partial}_q)_{q\ge 0}
\end{equation} 
as a chain complex of left-modules over the group ring $\k\pi$.  
We shall call $C_{\bullet}(\wX,\k)$ the {\em equivariant chain 
complex} of $X$, over $\k$.  

Using the  action of $\pi$ on $\wX$, we may identify 
\begin{equation}
\label{eq:tensor}
C_q(\wX,\k) \cong \k\pi \otimes_\k C_q(X,\k)
\end{equation}
as left $\k\pi$-modules.  Under this identification, 
a basis element of the form $1\otimes e$  from the right hand side 
corresponds to the unique lift $\tilde{e}$ of the $q$-cell $e$ 
sending the basepoint $*\in D^{q}$ to $\tilde{e}_0$. 

These constructions are functorial, in the following 
sense.  Suppose $f\colon X \to X'$ is 
a map between connected  CW-complexes.  
By cellular approximation, we may assume 
$f$ respects the CW-structures; in particular, 
$f(e_0)=e'_0$.  Let 
$f_*\colon C_{\bullet} (X,\k) \to C_{\bullet}(X',\k)$ 
be the induced map on cellular chain complexes, and let 
$f_{\sharp}\colon \pi \to \pi'$ 
be the induced homomorphism on fundamental groups. 
By covering space theory, the map $f$ lifts to a 
cellular map, $\tilde{f}\colon \wX \to \wX'$, 
uniquely specified by the requirement that 
$\tilde{f}(\tilde{e}_0)=\tilde{e}'_0$. The induced 
chain map, 
$\tilde{f}_*\colon C_{\bullet} (\wX,\k) \to C_{\bullet}(\wX',\k)$, 
is equivariant with respect to the ring homomorphism 
$\bar{f}_{\sharp} \colon \k\pi\to \k\pi'$. 

\subsection{The first differentials}
\label{subsec:diff}

Let $\epsilon \colon \k\pi \to \k$ be the augmentation homomorphism.   
Under identification \eqref{eq:tensor}, the induced homomorphism 
$p_*\colon C_{q}(\wX,\k)\to C_{q}(X,\k)$ coincides with 
$\epsilon \otimes \id \colon \k\pi \otimes_{\k} 
C_q(X,\k) \to  \k \otimes_{\k} C_q(X,\k)$.   

The first boundary map, $\tilde{\partial}_1$, is easy to write 
down.  Let $e_1$ be a $1$-cell of $X$.  Recall we assume  
$X$ has a single $0$-cell, $e_0$; thus, $e_1$ is a loop 
at $e_0$, representing an element $x=[e_1]\in \pi$.  
Let $\tilde{e}_1$ be the lift of $e_1$ at $\tilde{e}_0$.  Clearly,
$\tilde{\partial}_1(\tilde{e}_1)= (x-1) \tilde{e}_0$. Thus, 
$\tilde{\partial}_1\colon 
\k\pi \otimes_{\k} C_1(X,\k) \to \k\pi \otimes_{\k} C_0(X,\k)=\k\pi$ 
is given by
\begin{equation}
\label{eq:bdry1}
\tilde{\partial}_1(1 \otimes e_1)  = x-1. 
\end{equation}

The second boundary map, $\tilde{\partial}_2\colon 
\k\pi \otimes_{\k} C_2(X,\k) \to \k\pi \otimes_{\k} C_1(X,\k)$, 
can be written by means of Fox derivatives \cite{Fox}.  
If $\{e^i_1\}_i$ are the $1$-cells of $X$, and 
$e_2$ is a $2$-cell, then 
\begin{equation}
\label{eq:fox der}
\tilde{\partial}_2 (1\otimes e_2) = 
\sum_{i} \Big(\frac{\partial r}{\partial x_i}\Big)^{\phi} \otimes e^i_1,
\end{equation}
where $r$ is the word in the free group $F$ on generators $x_i$, 
determined by the attaching map of the $2$-cell, and 
$\bar{\phi}\colon \k{F} \to \k\pi$ is the extension to group 
rings of the projection map $\phi \colon F\surj \pi$. 
As for the higher boundary maps $\tilde{\partial}_{q}$, 
$q>2$, there is no general procedure for computing 
them, except in certain very specific situations.  

If $X$ is an Eilenberg-MacLane $K(\pi,1)$ space, i.e., if 
$\wX$ is contractible, then 
the augmented chain complex $\widetilde{C}_{\bullet} \colon 
C_{\bullet}(\wX,\k) \to \k$ is a free $\k\pi$-resolution of the 
trivial module $\k$.  For instance, if $T^n=K(\Z^n,1)$ is the $n$-torus, 
then  $\widetilde{C}_{\bullet}$ is the Koszul complex 
over the ring $\k\Z^n=\k[t_1^{\pm 1},\dots ,t_n^{\pm 1}]$.  

\subsection{Building up CW-complexes}
\label{subsec:discuss}

Given a group $\pi$, and a matrix $D$ over $\Z\pi$, it is possible 
to construct a CW-complex $X$ having $D$ as a block 
in one of the equivariant boundary maps in $C_{\bullet}(\wX,\Z)$.  

\begin{example}
\label{ex:construction}
Let $X_0$ be a connected CW-complex, with fundamental 
group $\pi= \pi_1(X_0,e_0)$, and  denote by $X_0\vee S^{n}$ 
the wedge sum of $X_0$ with the $n$-sphere, $n\ge 2$.  
Note that $\pi_1(X_0\vee S^n)=\pi$. The inclusion 
$S^n \inj X_0\vee S^n$ defines an element
$[S^n]\in \pi_n(X_0\vee S^n)$.

Given an element $x$ in $\Z\pi$, construct a new CW-complex, 
$X$, by attaching an $(n+1)$-cell along a map 
$\phi_x\colon S^n \to X_0\vee S^{n}$  
representing $x \cdot [S^n] \in \pi_n(X_0\vee S^n)$.  
Clearly, $\pi_1(X)=\pi_1(X_0)$, and 
$C_{\bullet}(\widetilde{X},\Z)= C_{\bullet}(\wX_0,\Z) \oplus C(n,x)$, 
where $C(n,x)$ denotes the 
elementary chain complex $C_{n+1}\to C_n$, with differential 
$\cdot x\colon \Z\pi\to \Z\pi$. 
\end{example}

More generally, for any integer $n\ge 2$, and any 
$\Z\pi$-linear map  $D\colon (\Z\pi)^{m} \to (\Z\pi)^{\ell}$, 
a construction much as above produces a CW-complex $X$, 
with $\pi_1(X)=\pi_1(X_0)$, and 
$C_{\bullet}(\wX,\Z)= C_{\bullet}(\wX_0,\Z) \oplus C(n,D)$, 
where $C(n,D)$ 
denotes the chain complex concentrated in degrees 
$n+1$ and $n$, with differential equal to $D$. 

\subsection{Coefficient modules}
\label{subsec:modules}

Let $X$ be a connected CW-complex, with fundamental 
group $\pi$. Suppose $M$ is a right $\k\pi$-module.  The cellular  
chain complex of $X$ with coefficients in $M$ is defined as 
\begin{align}
\label{eq:cxm}
C_{\bullet}(X, M)&=\big( C_q(X, M) ,\: \tilde{\partial}^M_q\big)_{q\ge 0} 
\\
\notag &=\big(M \otimes_{\k\pi} C_q(\wX, \k) ,\:  
\id_M  \otimes_{\k\pi}\, \tilde{\partial}_q\big)_{q\ge 0} . 
\end{align}
 In the particular case when $M$ is the free $\k\pi$-module 
of rank one, $C_{\bullet}(X, \k\pi)$ coincides with the 
equivariant chain complex $C_{\bullet}(\wX, \k)$.  

Similarly, if $M$ is a left $\k\pi$-module, define 
the cellular cochain complex of $X$ with 
coefficients in $M$ as 
\begin{align}
\label{eq:cochains}
C^{\bullet}(X, M)&=\big( C^q(X, M) ,\: \tilde{\delta}^q_M\big)_{q\ge 0} 
\\
\notag &=\big(\Hom_{\k\pi} (C_q(\wX, \k),M) ,\:  
\Hom_{\k\pi}( \tilde{\partial}_{q}, M)
\big)_{q\ge 0} . 
\end{align}

\begin{example}
\label{ex:locsyst}
Let $\k$ be a field, and let $\rho\colon \pi\to \GL(n,\k)$ be 
a linear representation of $\pi$.  Then $M=\k^n$ 
acquires a right module structure over $\k\pi$, 
via $mg = \rho(g^{-1})(m)$; such a module is called a 
rank $n$ local system on $X$, with monodromy $\rho$.  
We write $H_*(X, {}_{\rho}\k^n):=H_*(C_{\bullet}(X, M))$ 
for the homology of $X$ with coefficients in this local system. 
When $n=1$, there is no need to turn a representation into an
anti-representation; in this case, we simply define 
$H_*(X, \k_{\rho})$ to be $H_*(C_{\bullet}(X, M))$, where
$M=\k$ is viewed as a {\em right} $\k\pi$-module, via
$mg = \rho(g)(m)$.
\end{example}

\begin{example}
\label{ex:covers}
Let $Y\to X$ be a Galois cover, with group of deck transformations 
$G$ and classifying map $\nu\colon \pi\to G$.   The cell structure 
on $X$ lifts in standard fashion to a cell structure on $Y$.  
Let $\k{G}_{\nu}$ denote the group-ring of $G$, viewed as 
a right $\k\pi$-module, via $h\cdot g =h\nu(g)$.  
Similarly, let ${}_{\nu}\k{G}$ denote 
the same group-ring, viewed as a left $\k\pi$-module, via 
$g\cdot h =\nu(g) h$.
Then $C_{\bullet}(Y,\k)=C_{\bullet}(X,\k{G}_{\nu})$, 
as chain complexes of left $\k{G}$-modules, and 
$C^{\bullet}(Y,\k)=C^{\bullet}(X,{}_{\nu}\k{G})$, 
as cochain complexes of right $\k{G}$-modules.

Notable is the case of the universal abelian 
cover, $X^{\ab} \to X$, defined by the abelianization map,  
$\ab\colon \pi \surj \pi_{\ab}$. The homology groups 
$H_q(X^{\ab},\k)$, viewed as modules over the ring 
$\Lambda=\k\pi_{\ab}$, are called the {\em Alexander 
invariants} of $X$, while the boundary maps 
$\widetilde{\partial}^{\ab}_q= 
\id_{\Lambda} \otimes_{\k\pi} \widetilde{\partial}_q$ 
are called the {\em Alexander matrices} of $X$. 
\end{example}

\section{The equivariant spectral sequence}
\label{sec:spectral}

We now set up the spectral sequence associated to the 
$I$-adic filtration on the equivariant chain complex 
$C_{\bullet}(X, M)$, and analyze some of its properties. 

\subsection{A spectral sequence}
\label{subsec:ss}
We use  \cite{CE} and \cite{Sp} as standard references 
for spectral sequences. 
Given an increasing filtration $F_{\bullet}=\{F_{-n}\}_{n\ge 0}$  
on a chain complex $C_{\bullet}=(C_q,\partial_q)$ over a coefficient 
ring $\k$, there is a spectral sequence 
$E^{\bullet}=\{E^r_{s,t}, d^r\}_{r\ge 1}$. 
The $E^1$ term is defined as $E^1_{s,t}=H_{s+t}(F_s/F_{s-1})$, 
while the $d^1$ differential is the boundary operator in the homology 
exact sequence associated to the triple $(F_s,F_{s-1},F_{s-2})$. 
Each term $E^r$ is a bigraded $\k$-module, the differentials 
 $d^r$ have bidegree $(-r,r-1)$, and $E^{r+1}=H(E^r, d^r)$.  

Now let $X$ be a connected CW-complex as in \S\ref{subsec:equivchain}, 
with fundamental group $\pi=\pi_1(X)$, and augmentation 
ideal $I=I_{\k}(\pi)$. Let $M$ be a right $\k\pi$-module, 
and let $C_{\bullet}(X, M)$ be the cellular chain complex 
of $X$ with coefficients in $M$.

The $I$-adic filtration on $M$ yields a descending filtration, 
$F^0\supset F^1 \supset F^2 \supset \cdots$, 
on $C_{\bullet}(X, M)$.  The $n$-th term of this filtration is 
given by
\begin{equation}
\label{eq:filtcxm}
F^n(C_{\bullet}(X, M))= M\cdot I^n \otimes_{\k\pi} C_{\bullet}(\wX,\k).
\end{equation}
Clearly, the differentials $\tilde{\partial}^M$ of $C_{\bullet}(X, M)$ 
preserve this filtration. 
Set $F_{-n}  = F^n$.  Then $F_{\bullet}=\{F_{-n}\}_{n\ge 0}$ 
is an increasing filtration on $C_{\bullet}(X,M)$, bounded 
above by $F_{0}= C_{\bullet}(X,M)$.  

\begin{definition}
\label{def:ss}
The {\em equivariant spectral sequence} of the CW-complex 
$X$, with coefficients in the $\k\pi_1(X)$-module $M$ 
is the spectral sequence associated to the $I$-adic 
filtration $F_{\bullet}$ on the chain complex $C_{\bullet}(X,M)$,
\begin{equation*}
\label{eq:sscw}
\{E^r(X,M),d^r\}_{r\ge 0},
\end{equation*}
with differentials $d^r \colon E^r_{s,t} \to E^r_{s-r,t+r-1}$. 
\end{definition}

In order to analyze this spectral sequence, we need some 
preliminary facts.   

\subsection{The associated graded chain complex}
\label{subsec:asscc}
 Under the identification 
from \eqref{eq:tensor}, the terms of the chain 
complex \eqref{eq:cxm} can be written as
\begin{equation}
\label{eq:equiv again}
C_{q}(X,M)=M\otimes_{\k} C_{q}(X,\k).
\end{equation}
The $I$-adic filtration on this chain complex is then 
given by 
\begin{equation}
\label{eq:filtcxm again}
F^n(C_{\bullet}(X, M))= M\cdot I^n \otimes_{\k} C_{\bullet}(X,\k).
\end{equation}

Clearly, $F^n/F^{n+1}=\gr^n(M)\otimes_{\k} C_{\bullet}(X, \k)$, 
where  $\gr(M)=\bigoplus_{n\ge 0} M I^n / M I^{n+1}$. 
Now recall that the boundary maps in $C_{\bullet}(X,M)$ 
have the form $\tilde{\partial}^M=\id_M \otimes_{\k\pi} \tilde{\partial}$. 

\begin{lemma}
\label{lem:grd}
For each $q\ge 0$, we have 
$\gr(\id_M \otimes_{\k\pi} \tilde{\partial}_q)= 
\id_{\gr(M)} \otimes_{\k} \partial_q$.  
\end{lemma}

\begin{proof}
Let $e$ be a $q$-cell of $X$.  
Then 
\[
p_* \tilde{\partial}_q (1\otimes e)= \partial_q p_*  (1\otimes e) = 
 \partial_q(e)= p_* (1 \otimes \partial_q e).
\]
Hence, $\tilde{\partial}_q (1 \otimes e)  - 
1\otimes  \partial_q (e)$ belongs to $I\otimes_{\k} C_{q-1}(X,\k)$. 

Now let $m$ be an arbitrary element in $M\cdot  I^n$. 
By the above,
\[
\big ( \id_{M}\otimes_{\k \pi} \tilde{\partial}_q \big )
( m \otimes e ) - m \otimes \partial_q (e) \in 
M \cdot I^{n+1} \otimes_{\k} C_{q-1}(X,\k).
\]
The conclusion follows at once. 
\end{proof}

It follows that the graded chain complex 
associated  to filtration \eqref{eq:filtcxm again} has the form
\begin{equation}
\label{eq:grequiv}
\gr(C_{\bullet}(X,M))=(\gr(M) \otimes_\k C_q(X,\k),
\id_{\gr(M)} \otimes\, \partial_q)_{q\ge 0}. 
\end{equation} 
This chain complex is typically much easier to handle 
than the equivariant chain complex \eqref{eq:cxm}. 
Here is an illustration.

\begin{example}
\label{ex:grm easy}
Consider a local system $M$ as in Example \ref{ex:locsyst}.  
Suppose there is an element $g\in \pi$ such that $\rho(g)$ 
does not admit $1$ as an eigenvalue---this happens whenever 
$M$ is a non-trivial, rank $1$ local system.  Then clearly $MI=M$.  
Thus, $\gr(M)=0$ and $\gr(C_{\bullet}(X,M))$ is the 
zero complex in this instance.
\end{example}

\begin{remark}
\label{rem:grlin}
In the case when $M=\k\pi$, the $I$-adic filtration 
on $C_{\bullet}(X,M)=C_{\bullet}(\wX,\k)$ is simply 
$C_{\bullet}(\wX,\k) \supseteq I \cdot C_{\bullet}(\wX,\k)  
\supseteq \cdots  \supseteq I^n \cdot C_{\bullet}(\wX,\k) 
\supseteq \cdots$, 
while the associated graded chain complex takes the form
\begin{equation}
\label{eq:grcoeff}
\gr(C_{\bullet}(\wX,\k))=(\gr (\k\pi) \otimes_\k C_q(X,\k),
\id_{\gr (\k\pi)} \otimes\, \partial_q)_{q\ge 0}.
\end{equation} 
Notice that the differentials in this chain complex 
are $\gr(\k\pi)$-linear.   
\end{remark}

\subsection{The first pages}
\label{subsec:e0e1}

The $E^0$ term of the equivariant spectral sequence of $X$ 
with coefficients in $M$ is defined in the usual manner, as the 
associated graded of the filtration $F_{\bullet}$ on 
$C_{\bullet}(X,M)$. Using \eqref{eq:grequiv}, we find:   
\begin{equation}
\label{eq:e0}
E^0_{-p,q}(X,M)= \gr^{p} (M) \otimes_{\k} C_{q-p}(X,\k) ,
\end{equation}
if $p\ge 0$ and $q\ge p$, and otherwise $E^0_{-p,q}(X,M)=0$.
Hence, the spectral sequence is concentrated in the second quadrant.
Under this identification, and as a consequence of Lemma \ref{lem:grd}, 
the differential $d^0\colon E^0_{-p,q}(X,M) \to E^0_{-p,q-1}(X,M)$ 
takes the form $d^0= \id\otimes\, \partial$. Hence, 
\begin{equation}
\label{eq:e1}
E^1_{-p,q}(X,M)= H_{q-p}(X, \gr^{p} (M) ).
\end{equation}

\subsection{Functoriality properties}
\label{subsec:func}

Let $f\colon (X,e_0) \to (X',e'_0)$ be a cellular 
map, and suppose $\phi\colon M\to M'$ is  
equivariant with respect to $f_{\sharp}$. 
The resulting $\k$-linear map, 
\begin{equation}
\label{eq:phistar}
\phi \otimes \tilde{f}_* \colon C_{\bullet}(X, M)= 
M \otimes_{\k\pi} C_{\bullet}(\wX,\k) 
\to M' \otimes_{\k\pi'} C_{\bullet}(\wX',\k)= C_{\bullet}(X', M'), 
\end{equation}
is a chain map, preserving $I$-adic filtrations.  Consequently, 
$\phi \otimes \tilde{f}_*$ induces a morphism between 
the corresponding $I$-adic spectral sequences, 
\begin{equation}
\label{eq:morfismos}
E^r(\phi \otimes \tilde{f}_*)\colon E^r(X,M) \to E^r(X',M').
\end{equation}

In particular, $\phi \otimes \tilde{f}_*$ induces a chain map between 
the associated graded chain complexes,  
\begin{equation}
\label{eq:grftilde}
\gr(\phi \otimes \tilde{f}_*)\colon 
\gr(M)\otimes_{\k} C_{\bullet}(X,\k) \to 
\gr(M') \otimes_{\k}  C_{\bullet}(X',\k).
\end{equation}

\begin{lemma}
\label{lem:grf}
For each $q\ge 0$, we have 
$\gr(\phi \otimes \tilde{f}_q)=\gr(\phi ) \otimes f_q$.
\end{lemma}

\begin{proof}
It is readily seen that, under the identification 
$\gr^0(\k\pi) =\k$, the map $\gr^0(\tilde{f}_q)$ 
corresponds to the map $f_q\colon C_q(X,\k)\to  C_q(X',\k)$.   

Let $m$ be an element of $MI^n$, and let $e$ be a 
$q$-cell of $X$, with lift $\tilde{e}$. By the above, 
\[
\tilde{f}_q(\tilde{e}) - 1\otimes f_q(e)\in I'\cdot C_q(\wX',\k).
\]
Hence,  $(\phi \otimes \tilde{f}_q) (m\otimes \tilde{e})=  
\phi(m)\otimes \tilde{f}_q (\tilde{e})$ equals  
$\phi(m)\otimes f_q(e)$ modulo $F'^{n+1}$. 
The conclusion follows.
\end{proof}

As a simple example, take $f=\id_X$.  Then any morphism 
$\phi\colon M\to M'$ of right $\k\pi$-modules induces 
a morphism of spectral sequences, $E^{\bullet}(\phi)\colon 
E^{\bullet}(X,M)\to E^{\bullet}(X,M')$, with 
$E^0(\phi)=\gr(\phi)\otimes \id$. 

\subsection{Homotopy invariance}
\label{subsec:hinv}

Let $f\colon (X,e_0)\to (X',e'_0)$ be a cellular homotopy 
equivalence, and let $M$ be a right $\k\pi$-module. 
Use the isomorphism $f_{\sharp}\colon \pi\to \pi'$ to 
view $M$ as a right $\k\pi'$-module. 

\begin{corollary}
\label{cor:uniqueness}
The spectral sequences $\{E^r(X,M)\}_{r\ge 1}$ and 
$\{E^r(X',M)\}_{r\ge 1}$ are isomorphic. 
\end{corollary}

\begin{proof}
View $\id_M$ as an equivariant map with respect 
to $f_{\sharp}$. Using Lemma \ref{lem:grf} and 
identification \eqref{eq:e1}, we see that 
the map $E^1(\id_M \otimes \tilde{f}_*)$ 
coincides with the  induced homomorphism 
$H_*(f)\colon  H_*(X,\gr(M)) \to H_*(X',\gr(M))$. 
Hence, the maps $E^r(\id_M \otimes \tilde{f}_*)$ 
are isomorphisms, for all $r\ge 1$.
\end{proof}

\begin{corollary}
\label{cor:cw space}
Let $X$ be a path-connected topological space having the 
homotopy type of a CW-complex, and let $M$ be a 
right $\k\pi_1(X)$-module. Then there is a well-defined 
second quadrant $I$-adic spectral sequence 
$\{ E^r (X,M), d^r \}$, starting at $r=1$. 
\end{corollary}

\section{Identifying the $d^1$ differential}
\label{sec:e1 page}

Throughout this section, $X$ 
is a CW-complex with a single $0$-cell $e_0$ and 
with basepoint-preserving attaching maps, 
$\pi=\pi_1(X,e_0)$ is the fundamental group, and 
$C_{\bullet}(X,M)$ is the cellular chain complex,  
with coefficients in a right $\k\pi$-module $M$. 
Our goal is to identify the first page of the equivariant 
spectral sequence of $X$, in terms of the coalgebra   
structure of $H_*(X,\k)$ and the module structure of $M$. 

\subsection{The first page}
\label{subsec:e1 page} 
Let $\{ E^r (X,M),d^r\}_{r\ge 1}$ be the spectral sequence 
from Definition \ref{def:ss}. We wish to analyze the differentials 
\begin{equation}
\label{eq:dq diff}
d^1\colon E^1_{-p,q}(X,M) \to E^1_{-p-1,q}(X,M), 
\end{equation}
for all $p\ge 0$ and $q\ge p$. Schematically, the $E^1$ page 
looks like
\[
\xymatrixrowsep{6pt}
\xymatrix{
E^1_{-2,2} &  E^1_{-1,2} \ar_(.4){d^1}[l] & E^1_{0,2} 
\ar_(.4){d^1}[l]\\
&  E^1_{-1,1} & E^1_{0,1} \ar_(.4){d^1}[l]\\
& & E^1_{0,0}
}
\]

To identify these maps in terms of cohomological data, we need to 
assume the following:  Either $\k=\Z$ and $H_*(X,\Z)$ is torsion-free, 
or $\k$ is a field.  Using this assumption, formula \eqref{eq:e1}, 
and the Universal Coefficients Theorem, we find
\begin{equation}
\label{eq:e1 new}
E^1_{-p,q}(X,M)= \gr^{p} (M) \otimes_{\k} H_{q-p}(X,\k), 
\end{equation}
for each $p\ge 0$ and $q\ge p$.  

The $d^1$ differentials enjoy the following naturality property. 
Suppose $f\colon X\to X'$ is a cellular map, and $\phi\colon M\to M'$ 
is a morphism of modules, equivariant with respect to 
$\bar{f}_{\sharp}\colon \k\pi_1(X)\to \k\pi_1(X')$.  
From Lemmas \ref{lem:grd} and \ref{lem:grf}, 
we deduce $E^1(\phi \otimes \tilde{f}_*)= \gr(\phi) \otimes H_*(f)$. 
Using the naturality property from \eqref{eq:morfismos}, we 
obtain a commuting diagram,
\begin{equation}
\label{eq:E1nat}
\xymatrix{
\gr^p(M) \otimes_{\k} H_q(X,\k) \ar[rr]^{d^1}
\ar[d]^{\gr(\phi) \otimes H_*(f)} 
&& \gr^{p+1}(M) \otimes_{\k} H_{q-1}(X,\k) 
\ar[d]^{\gr(\phi)\otimes H_*(f)}   \\
\gr^p(M') \otimes_{\k} H_q(X',\k) \ar[rr]^{d'^1}
&& \gr^{p+1}(M') \otimes_{\k} H_{q-1}(X',\k) 
}
\end{equation}

\subsection{Cohomological interpretation}
\label{subsec:d1diff} 
The main result of this section is the following Theorem, 
which identifies the differentials on the $E^1$ page in terms 
of the comultiplication in $H_*(X, \k)$, 
and the $\k\pi$-module structure on $M$, given by the 
multiplication map, $\mu_M\colon M\otimes_{\k} \k\pi \to M$, 
$m \otimes g\mapsto mg$.  Fix integers $p\ge 0$ and $q\ge 1$.  

\begin{theorem}
\label{thm:d1map}
Let $X$ be a connected CW-complex, and $M$ a 
right $\k\pi$-module.  Assume either $\k=\Z$ and $H_*(X,\Z)$ 
is torsion-free, or $\k$ is a field. Then the differential 
$d^1\colon E^1_{-p,p+q}(X,M) \to E^1_{-p-1,p+q}(X,M)$ 
can be decomposed as
\[
\xymatrix{
\gr^p(M) \otimes_{\k} H_q(X,\k) \ar^{d^1}[r]  
\ar^{\id\otimes \nabla_X}[d]&
\gr^{p+1}(M) \otimes_{\k} H_{q-1}(X,\k) \\
\gr^p(M) \otimes_{\k} (H_1(X,\k)  \otimes_{\k} H_{q-1}(X,\k) )
\ar^{\cong}[r]
&( \gr^p(M) \otimes_{\k} \gr^1(\k\pi) ) \otimes_{\k} H_{q-1}(X,\k)\, . 
\ar_{\gr(\mu_M)\otimes \id }[u]
}
\]
\end{theorem}

Here, $\nabla_X$ is the comultiplication map, defined 
as the composite 
\begin{equation*}
\label{eq:comult}
\xymatrix{
H_q(X,\k) \ar[r]^(.46){\Delta_*}  \ar[drr]_{\nabla_X}
& H_q(X\times X,\k) \ar[r]^(.36){\cong} & 
\DS{\bigoplus_{i=0}^{q} H_i(X,\k) \otimes_{\k} H_{q-i}(X,\k)}
\ar[d]^(.53){\proj} \\
&& H_1(X,\k) \otimes_{\k} H_{q-1}(X,\k),
}
\end{equation*}
where the first arrow is the homomorphism induced by the 
diagonal map, the second arrow is the K\"unneth isomorphism, 
and the third arrow is projection onto direct summand.
Under the identification $H_i(X,\k)^* \cong H^i(X,\k)$, 
the map $\nabla_X$ is the dual of the cup-product map 
$\cup_X \colon H^1(X,\k)\otimes_{\k} H^{q-1}(X,\k) \to H^q(X,\k)$, 
provided $X$ is a finite-type CW-complex. 

The proof of Theorem \ref{thm:d1map} will occupy the rest 
of this section. 

\subsection{Reducing to the case $M=\k\pi$}
\label{subsec:mtor}

Using the functoriality of the spectral sequence with 
respect to coefficient modules, we first reduce the proof 
to the case when $M$ is a free $\k\pi$-module of rank $1$.

\begin{lemma}
\label{lem:reduction}
If the conclusion of Theorem  \ref{thm:d1map} holds for the 
coefficient module $\k\pi$, then it holds for any coefficient 
module $M$.
\end{lemma}

\begin{proof}
Let $\psi\colon N \to M$ be a homomorphism of right 
$\k\pi$-modules. Write $H_q=H_q(X,\k)$, and 
$\otimes=\otimes_{\k}$. We then have the following 
cube diagram:
\begin{equation*}
\label{eq:cdm}
\xymatrixcolsep{-20pt}
\xymatrixrowsep{32pt}
\xymatrix{
& \gr^p(N) \otimes (H_1 \otimes H_{q-1}) \ar^(.43){\cong}[rr]  
 \ar'[d]^(.5){\gr(\psi) \otimes \id}[dd] 
&& ( \gr^p(N) \otimes \gr^1(\k\pi)) \otimes H_{q-1}
 \ar^{\gr(\mu_N)\otimes \id}[dl] 
 \ar^(.35){\gr(\psi) \otimes \id}[dd]
\\
\gr^p(N) \otimes H_q \ar^(.52){d^1_{N}}[rr]  
\ar^(.35){\gr(\psi) \otimes \id}[dd] 
\ar^{\id\otimes \nabla}[ur] &&
\gr^{p+1}(N) \otimes H_{q-1} 
\ar^(.35){\gr(\psi) \otimes \id}[dd]
\\
& \gr^p(M) \otimes (H_1 \otimes H_{q-1}) \ar'[r]^(.9){\cong}[rr]  
&& ( \gr^p(M) \otimes \gr^1(\k\pi)) \otimes H_{q-1}
 \ar^{\gr(\mu_M)\otimes \id}[dl]
\\
\gr^p(M) \otimes  H_q \ar^(.52){d^1_{M}}[rr]  
\ar^{\id\otimes \nabla}[ur] 
&&\gr^{p+1}(M) \otimes H_{q-1}
}
\end{equation*}
All side squares of the cube commute:  the front one 
by naturality of $d^1$, as explained in \eqref{eq:E1nat}, 
the right one by functoriality of $\gr$, and the other 
two for obvious reasons. 

Now suppose the top square commutes for $N=\k\pi$. 
Let $[x] \otimes h$ be an additive generator of
$\gr^p(M) \otimes H_q$, where $x=y a$, 
with $y\in M$ and $a\in I^p$. Let $\psi\colon \k\pi \to M$ 
be the $\k\pi$-linear map sending the unit of $\pi$ to $y$.  
Chasing the cube diagram, we see that the bottom 
square commutes, with 
$[x] \otimes h=(\gr(\psi)\otimes \id)([a]\otimes h) $ 
as input. Hence, the bottom square commutes. 
\end{proof}

\subsection{The case $M=\k\pi$}
\label{subsec:sskpi}

The $E^1$-term of the $I$-adic spectral sequence 
of $X$ with coefficients in $\k\pi$, 
\[
E^1(X,\k\pi)=\gr(\k\pi) \otimes_{\k} H_*(X,\k),
\]
is a left $\gr(\k\pi)$-module, freely generated by 
a $\k$-basis for $1 \otimes H_*(X,\k)$. We want to show 
\begin{equation}
\label{eq:d1 nabla}
d^1=(\gr(\mu_{\k\pi})\otimes \id) \circ 
(\id \otimes \nabla_X).
\end{equation}
Clearly, the map on the right side is $\gr(\k\pi)$-linear. In the next 
Lemma, we verify that $d^1$ is $\gr(\k\pi)$-linear, too (a more 
general result will be proved in Lemma \ref{lem:dr lin}). 

\begin{lemma}
\label{lem:d1lin}
The differential $d^1\colon E^1_{-p,p+q}(X,\k\pi) \to 
E^1_{-p-1,p+q}(X,\k\pi)$ is $\gr(\k\pi)$-linear. 
\end{lemma}

\begin{proof}
Let $F^s= I^s C_{\bullet}(\wX,\k)$ be the $I$-adic filtration 
on $C_{\bullet}(X, \k\pi)$. By definition, the differential $d^1$ 
is the connecting homomorphism in the homology exact 
sequence of 
\[
\xymatrix{0\ar[r]& F^{s+1}/F^{s+2} \ar[r]& F^{s}/F^{s+2} 
\ar[r]& F^{s}/F^{s+1} \ar[r]& 0}.
\]

Let  $z\in C_q(X,\k)$ such that $\partial_q(z)=0$.  Then
$d^1(1\otimes [z]) \equiv \tilde{\partial}(1\otimes z),\, 
\bmod\, F^2$. 
More generally, if $x\in I^p$, then
$d^1([x]\otimes [z]) \equiv \tilde{\partial}(x\otimes z),\, 
\bmod\, F^{p+2}$. 
On the other hand,  $\tilde{\partial}(x\otimes z)= x
\tilde{\partial} (1\otimes z)$. Combining these formulas  
implies $d^1([x]\otimes [z]) = [x] d^1(1\otimes [z])$. 
Clearly, this implies the claim of the lemma.
\end{proof}

Thus, to verify \eqref{eq:d1 nabla}, it is sufficient to check 
equality on free $\gr(\k\pi)$-generators, i.e., on a $\k$-basis 
for $1 \otimes H_*(X,\k)$.  In other words, to identify the 
differentials on the $E^1$ page, we only need to show that, 
upon identifying  $E^1_{0,q}=H_q(X,\k)$ and 
$E^1_{-1,q}= H_1(X,\k)  \otimes_{\k} H_{q-1}(X,\k)$, the map 
$d^1\colon  E^1_{0,q}\to E^1_{-1,q}$ coincides with the 
comultiplication map $\nabla_X$, for all $q\ge 1$.

\subsection{Further reductions}
\label{subsec:further reduce}
Clearly, we may assume $X$ is a finite CW-complex. 
Indeed, an arbitrary homology class in $H_q(X,\k)$ is represented 
by a cycle supported on a finite subcomplex of $X$. Dualizing, 
we are left with proving the following Proposition.

\begin{prop}
\label{prop:deltacup}
Let $X$ be a connected, finite-type CW-complex, and assume 
either $\k=\Z$ and $H_*(X,\Z)$ is torsion-free, or $\k$ is a field. 
Then, the dual $\delta^1=(d^1)^{*}\colon  (E^1_{-1,q})^* \to (E^1_{0,q})^*$ 
coincides with the cup-product map $\cup_X\colon  
 H^1(X,\k)  \otimes_{\k} H^{q-1}(X,\k) \to H^{q}(X,\k)$.
\end{prop}

In the above, we may suppose $\k$ is actually a prime field. 
Indeed, in the first case, both $\delta^1$ and $\cup_X$ 
are defined over $\Z$, and $H^*(X, \Z)$ injects into 
$H^*(X, \Q)$, so we may replace $\k=\Z$ by $\Q$.  
In the second case, both maps are defined over the 
prime field of $\k$, so we may replace $\k$ by its 
prime field.

We will reduce the proof to a special class of spaces, 
using the functoriality properties of the spectral sequence.  
Let us first state those properties, in the form we will need. 

Let $f\colon X\to X'$ be a cellular map, and consider 
diagram \eqref{eq:E1nat}, for the morphism 
$\phi=\bar{f}_{\sharp}\colon \k\pi\to \k\pi'$.  In degree 
$p=0$, this diagram simplifies to
\begin{equation}
\label{eq:e1d0}
\xymatrixcolsep{50pt}
\xymatrix{
H_q(X,\k) \ar[r]^(.4){d^1}
\ar[d]^{H_*(f)} 
& H_1(X,\k) \otimes_{\k} H_{q-1}(X,\k) 
\ar[d]^{H_*(f)\otimes H_*(f)}   \\
H_q(X',\k) \ar[r]^(.4){d'^1}
& H_1(X',\k) \otimes_{\k} H_{q-1}(X',\k) 
}
\end{equation}
Dualizing, and writing $f^*=H^*(f)$, we obtain the 
commuting diagram
\begin{equation}
\label{eq:e1cup}
\xymatrixcolsep{45pt}
\xymatrix{
H^1(X,\k) \otimes_{\k} H^{q-1}(X,\k)  
\ar[r]^(.6){\delta^1}  &H^q(X,\k)
\\
H^1(X',\k) \otimes_{\k} H^{q-1}(X',\k)  
\ar[u]_{f^{*}\otimes f^{*}}  
\ar[r]^(.65){\delta'^1} &H^q(X',\k)
\ar[u]_{f^*}
}
\end{equation}

\subsection{A computation with Eilenberg-MacLane spaces}
\label{subsec:kpin}
Recall $\k$ is a prime field. Let $R=\Z$ if $\k=\Q$, and 
$R=\Z_p$ if $\k=\F_p$. For an integer $r\ge 1$, let 
$K_r$  be an Eilenberg-MacLane space of type $K(R,r)$. 
We may assume $K_r$ is a finite-type CW-complex, 
obtained from $S^r$ by attaching cells of dimension 
$r+1$ and higher, 
\[
K_r=e_0\cup e_r \cup e^{1}_{r+1}\cup \cdots \cup e^{m}_{r+1} 
\cup \cdots
\] 
Let $\brac{e_r}\in H_r(K_r,R)$ be the homology class 
represented by the $r$-cell, and let $\brac{e_r}^*\in  H^r(K_r,\k)$ 
be its dual over $\k$. 

Now let $K'_r$ be another copy of $K_r$, with cells $e'_j$, 
and consider the product CW-complex $K=K_1\times K'_{r}$. 
Let $\proj$ and $\proj'$ be the projections of $K$ onto the 
two factors.  Define cohomology classes 
$u_1=\proj^*(\brac{e_1}^*) \in H^1(K,\k)$ and 
$u'_r=\proj'^*(\brac{e'_r}^*)  \in H^r(K,\k)$.

\begin{lemma} 
\label{lem:del cup}
For the CW-complex $K$ defined above, 
$\delta^1(u_1\otimes u'_{r}) =  u_1 \cup u'_{r}$.
 \end{lemma}
 
\begin{proof}
Consider the map $\delta^1\colon 
H^1(K) \otimes H^{r}(K) \to H^{r+1}(K)$, 
with $\k$ coefficients. By Kronecker duality, 
we only need to verify:
\begin{equation}
\label{eq:kron}
\langle d^1(z), u_1\otimes u'_{r} \rangle = 
 \langle z, u_1\cup u'_{r} \rangle, 
\end{equation}
for all $z\in H_{r+1}(K)$.  By the K\"{u}nneth formula, 
\begin{equation}
\label{eq:kunneth}
H_{r+1}(K)=(H_0(K_1)\otimes H_{r+1}(K'_{r})) \oplus 
(H_1(K_1)\otimes H_{r}(K'_{r}))\oplus 
(H_{r+1}(K_1)\otimes H_{0}(K'_{r})).
\end{equation}
By construction, 
$u_1\cup u'_{r} = \brac{e_1}^* \times \brac{e'_r}^*$.  
It follows that
$ \langle z, u_1\cup u'_{r} \rangle= 1$ if 
$z=\brac{e_1} \times \brac{e'_r}$ and vanishes 
if $z$ belongs to one of the other two summands in 
\eqref{eq:kunneth}. 

Similarly, the left hand side of \eqref{eq:kron} 
vanishes if $z$ belongs to one of those two summands; 
this follows from the naturality of $d^1$, as expressed 
in \eqref{eq:e1d0}. 

Now suppose $r>1$. Note that $\pi_1(K_1)$ is a 
cyclic group, generated by $[e_1]$, whereas 
$\pi_1(K'_{r})=0$.  Moreover, 
$\widetilde{K}=\widetilde{K}_1 \times K'_{r}$. 
We compute:
\[
\tilde{\partial} (\tilde{e}_1\times e'_{r}) = 
\tilde{\partial}(\tilde{e}_1)\times e'_{r} - 
\tilde{e}_1\times \partial(e'_{r}) = 
([e_1]-1) \tilde{e}_0 \times e'_{r}, 
\]  
and so $d^1(\brac{e_1\times e'_{r}}) = 
\brac{e_1\times e'_0} \otimes \brac{e_0\times e'_{r}}$. It follows that 
$\langle d^1(\brac{e_1\times e'_{r}}), u_1\otimes u'_{r} \rangle =1$.

If $r=1$, then $\widetilde{K}=\widetilde{K}_1 
\times \widetilde{K}'_{1}$.  We compute:
\[
\tilde{\partial} (\tilde{e}_1\times \tilde{e}'_1) = 
\tilde{\partial}(\tilde{e}_1)\times \tilde{e}'_1 - 
\tilde{e}_1\times \tilde{\partial}(\tilde{e}'_1) = 
([e_1]-1) \tilde{e}_0 \times \tilde{e}'_1 
-([e'_1]-1) \tilde{e}_1 \times \tilde{e}'_0, 
\]  
and so $d^1(\brac{e_1\times e'_1}) = 
\brac{e_1\times e'_0} \otimes \brac{e_0\times e'_{1}} - 
\brac{e_0\times e'_1} \otimes \brac{e_1\times e'_0}$. 
It follows that 
$\langle d^1(\brac{e_1\times e'_{1}}), u_1\otimes u'_{1} \rangle =1$.
\end{proof}

\subsection{Proof of Proposition \ref{prop:deltacup}}
\label{subsec:proofdeltacup}

First assume $q=1$.  Let $e$ be a $1$-cell 
of $X$, and  $\brac{e}\in H_1(X,\k)$ the homology 
class it represents.  From the identification $I/I^2=\gr^1(\k\pi)$, 
and formula \eqref{eq:bdry1}, we get $d^1(\brac{e})=\brac{e}$. 
Thus,  $d^1\colon  E^1_{0,1}\to E^1_{-1,1}$ coincides with 
$\nabla_X=\id\colon  H_{1}(X,\k) \to H_{1}(X,\k)$.

Now assume $q>1$.  It is enough to show that 
$\delta^1(v_1 \otimes v_{q-1}) = v_1 \cup v_{q-1}$ over $\k$, 
for each $v_1\in H^1(X,R)$ and $v_{q-1}\in H^{q-1}(X,R)$.  
Let $K_r=K(R,r)$ be an Eilenberg-MacLane space as above. 
By obstruction theory, there is a  map $f_r\colon X\to K_r$ 
such that $v_r=f_r^*(\brac{e_r}^*)$. 
Now set $K=K_1\times K'_{q-1}$, where $K'_{q-1}$ 
is a copy of $K_{q-1}$, and define $f$ to be a cellular approximation 
of the map $F=(f_1,f'_{q-1})\colon X\to K$.  With notation as above, 
$v_1=f^*(u_1)$ and $v_{q-1}=f^*(u'_{q-1})$, over $\k$.
Using diagram \eqref{eq:e1cup} and the naturality of 
cup-products, the conclusion follows from Lemma \ref{lem:del cup}.
\hfill\qed

\section{Convergence issues}
\label{sec:convergence}

In this section, we discuss the convergence properties 
of the equivariant spectral sequence $E^{\bullet}(X,M)$.  
Throughout this section, $\k$ will denote a fixed field. 

\subsection{The $E^{\infty}$ term}
\label{subsec:conv}
Let $C_{\bullet}=(C_q,\partial_q)$ be a chain complex 
over $\k$,  endowed with an increasing filtration, 
$\cdots \subset F_{-2}\subset F_{-1} \subset F_{0}=C$, 
and let $E^{\bullet}=\{E^r_{s,t}, d^r\}_{r\ge 1}$ be the spectral 
sequence associated to $F_{\bullet}$.  Suppose 
the following condition holds: For each $s$ and $t$, 
the $\k$-vector space $E^1_{s,t}$ is finite-dimensional. Then
$E^r_{s,t}=E^{r+1}_{s,t}$  for all $r\ge r_0$. 
In this case, the $E^{\infty}$ term is defined as 
$E^{\infty}_{s,t}=E^{r_0}_{s,t}$, and the inclusion
\begin{equation}
\label{eq:inftyb}
E^{\infty}_{s,t} \supseteq \gr^s (H_{s+t}(C))\, ,
\end{equation}
holds for all $s, t$.

The spectral sequence is said to be {\em convergent} 
if $E^{\infty}_{s,t} = \gr^s (H_{s+t}(C))$, for all $s, t$.  
For short, we write $E^1_{s,t}\Rightarrow H_{s+t}(C)$. 
Recall $Z^r_s=\{z\in F_s \mid \partial z \in F_{s-r}\}$, and 
write $Z^{\infty}_s=\{z\in F_s \mid  \partial z =0\}$. 
In this setup, convergence of $E^{\bullet}$ is equivalent to
\begin{equation}
\label{eq:weak conv}
\bigcap_{r} (Z^r_s + F_{s-1}) \subseteq Z^{\infty}_s + F_{s-1}, 
\quad \text{for all $s$.}
\end{equation}

\subsection{The $E^{\infty}$ term of the equivariant spectral sequence}
\label{subsec:Iadic conv}
Now let $X$ be a connected CW-complex, with fundamental group 
$\pi=\pi_1(X)$.  Let  $I=I_\k(\pi)$ be the augmentation ideal of $\k\pi$, 
and let $M$ be a right $\k\pi$-module.  Under fairly general assumptions, 
the $I$-adic spectral sequence $E^{\bullet}(X,M)$ has an $E^{\infty}$ term.  
Its convergence, though, is a rather delicate matter. 

\begin{prop}
\label{prop:conv ss}
Suppose $\dim_{\k} H_q(X,\k)<\infty$, for all $q\ge 0$, 
and $\dim_{\k} M/MI < \infty$. Then, the spectral 
sequence $E^{\bullet}(X,M)$ has an $E^{\infty}$ term. 
\end{prop}

\begin{proof}
Recall $\gr(M)$ is generated as a $\gr(\k\pi)$-module 
by $\gr^0(M)=M/MI$, and $\gr (\k \pi)$ is generated as a ring by
$\gr^1 (\k \pi)= H_1(X, \k)$. We infer
that $\dim_{\k} \gr^s(M)< \infty$, 
for all $s$. Hence, for any fixed $s$ and $t$, the 
vector space $E^1_{-s,t}= \gr^s(M)\otimes_{\k} H_{t-s}(X,\k)$ 
is finite-dimensional.  The conclusion follows.
\end{proof}

\begin{remark}
\label{rem:einf}
Let $X$ be a finite-type CW-complex, and suppose $M=\k\pi$.  
Then $M/MI=\k$, and so the $E^{\infty}$ term exists. 
More generally, suppose $\nu\colon \pi\surj G$ is an 
epimorphism. Take $M=\k{G}_{\nu}$, and let $J=I_\k{G}$.  
Then $M/MI=\k{G}/J=\k$, and so $E^{\infty}$ exists.
\end{remark}

\subsection{A non-convergent spectral sequence}
\label{subsec:non conv}
We now give an example of a CW-complex $X$ and 
a $\k\pi$-module $M$ for which 
the assumptions of Proposition \ref{prop:conv ss} 
are satisfied, yet the spectral sequence 
$E^{\bullet}(X,M)$ does not converge.  More precisely, 
for each integer $m\ge 3$, we construct a finite CW-complex 
$X$ of dimension $m$ for which the convergence 
condition \eqref{eq:weak conv} fails for the coefficient 
module $M=\k\pi$, in filtration degree $s=0$ and 
total degree $m$.  

\begin{figure}[t]
\centering
\includegraphics[height=1.8in,width=1.8in]{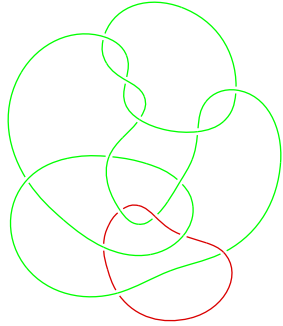} 
\caption{Hillman's $2$-component link}
\label{fig:hillman link}
\end{figure}

\begin{example}
\label{ex:non conv}
Let $L$ be the $2$-component link in $S^3$ depicted in 
Figure \ref{fig:hillman link},\footnote{%
This picture was drawn with the help of the Mathematica 
package {\sf KnotTheory}, by Dror Bar-Natan, and the 
graphical package {\sf Knotilus}, by Ortho Flint and Stuart Rankin.}
and let $\pi=\pi_1(S^3\setminus L)$ 
be the fundamental group of its complement.  In \cite{Hil}, 
J.~Hillman showed that this group is not residually nilpotent; 
that is, if  $\{\Gamma_n(\pi)\}_{n\ge 1}$ 
is the lower central series of $\pi$, then 
$\Gamma_{\omega}(\pi):=\bigcap_{n\ge 1} \Gamma_n(\pi)$ 
is non-trivial.  Pick an element 
$1\ne g\in \Gamma_{\omega}(\pi)$, and set 
$x=g-1\in I$.  From \cite{Q}, we know that 
the map  $\pi\inj \k\pi$, $h\mapsto h-1$, sends 
$\Gamma_n(\pi)$ to $I^n$, for all $n\ge 1$.  
Hence, $x$ belongs to the $I$-adic radical,  
$I^{\omega}:=\bigcap_{n\ge 1} I^n$. 

The link complement has the homotopy type of a 
finite $2$-complex, say $X_0$. Fix an integer $m\ge 3$.   
Using the construction from Example \ref{ex:construction}, 
we may define a CW-complex 
\[
X=(X_0\vee S^{m-1} ) \cup_{\phi_x} e_{m},
\] 
with $[\phi_x]=x\cdot [S^{m-1}] \in \pi_{m-1}(X_0\vee S^{m-1})$. 
Clearly,  $\pi_1(X)=\pi$. We know that
$C_{\bullet}(\wX,\k)= C_{\bullet}(\wX_0,\k)\oplus C(m-1, x)$, and
the differential of $C(m-1, x)$,
$\tilde{\partial}_m \colon \k \pi \to \k \pi$ ,
sends $1\in \k\pi$ to $x \in \k\pi$.  

From the definitions,  
$Z^r_{0,m}=\{ z \in \k\pi\mid \tilde{\partial}_m(z) \in I^r\}$; 
hence,  $1\in \bigcap_{r} Z^r_{0,m}$.  On the other 
hand, $Z^{\infty}_{0,m} =\ker (\tilde{\partial}_m)$. 
Thus, if \eqref{eq:weak conv} were to hold, we 
would  have 
\[
1\in \ker (\tilde{\partial}_m)+I.
\]

Now, $\ker (\tilde{\partial}_m)=\{y \in \k\pi \mid yx=0\}$, and 
this subspace vanishes. Indeed, every link group is locally 
indicable, by \cite{HSh}, and the group algebra of a locally 
indicable group has no zero divisors, by \cite{Hig}.  
We conclude that $1\in I$, a contradiction.
\end{example}

\section{Base change and filtrations on homology}
\label{sect:change rings}

We now turn to the special case when the module $M$ 
is actually a ring $R$, with right $\k\pi$-module structure given 
by extension of scalars.  We analyze the $d^1$ differential 
of $E^1(X,R)$, and two natural filtrations on $H_*(X,R)$. 

\subsection{Base change}
\label{subsec:base change}
Let $\pi$ be a group, and $\k$ a commutative, unital ring. Suppose we 
are given a ring $R$, and a ring homomorphism $\rho\colon \k\pi\to R$ 
(also known as a `base change' or `extension of scalars'). 
Then $R$ becomes a right $\k\pi$-module by setting $r\cdot x =r \rho(x)$, 
for all $x\in \k\pi$ and $r\in R$.  We will denote by $J$ the two-sided 
ideal of $R$ generated by $\rho(I)$, where $I$ is the augmentation ideal 
of $\k\pi$.  

A particular case arises when we are given a group homomorphism, 
$\nu\colon \pi\to G$.  Then, the linear extension of $\nu$ to group 
rings, call it $\bar{\nu}\colon \k\pi \to \k{G}$, is a ring homomorphism. 
As in Example \ref{ex:covers}, we will denote the resulting 
$\k{\pi}$-module by $\k{G}_{\nu}$. 

\begin{lemma}
\label{lem:irj}
With notation as above, suppose $R$ is a commutative ring, 
or $R=\k{G}_{\nu}$, for some epimorphism $\nu\colon \pi\surj G$.  
Then:
\begin{enumerate}
\item \label{i1} 
$R I^{n}=J^{n}$, for all $n\ge 0$. If $R=\k{G}_{\nu}$, then $J=I_\k{G}$. 
\item \label{i2}
$\gr_I(R)=\gr_J(R)$.
\end{enumerate}
\end{lemma}

\begin{proof}
\eqref{i1} By definition, $J=R \rho(I) R$, and so the inclusion 
$R I^{n}\subset J^{n}$ always holds.  If $R$ is commutative, the 
reverse inclusion is obvious. 

Now suppose $R=\k{G}_{\nu}$. Recall $I=I_\k\pi$ is generated 
(as a $\k$-module) by all elements of the form $g-1$, with $g\in \pi$. 
Since $\nu$ is surjective, it follows that $\bar{\nu} (I)=I_\k{G}$.  
In particular, $\bar{\nu} (I)$ is a two-sided ideal of $R$, so 
$J=I_\k{G}$ and $J^{n} \subset R I^{n}$. 

\eqref{i2} Follows from \eqref{i1}.
\end{proof}

\subsection{The spectral sequence with $R$-coefficients}
\label{subsec:exr}  
Let $X$ be a connected CW-complex with a single 
$0$-cell $e_0$, and fundamental group $\pi=\pi_1(X,e_0)$. 
Suppose $R$ is a ring, endowed with the right $\k\pi$-module 
structure given by a ring homomorphism  $\rho\colon \k\pi\to R$. 
Then the $I$-adic spectral sequence $E^{\bullet}(X, R)$
is a spectral sequence in the category of left $R$-modules. 
Indeed, $R$ acts on itself by left-multiplication, and this 
action extends in a natural way to each term $E^{r}(X, R)$. 
Furthermore, all differentials $d^r_R\colon E^{r} \to E^{r}$ 
are $R$-linear. 
  
Now assume $R$ satisfies one of the two conditions of 
Lemma~\ref{lem:irj}. As above, let $J$ be the two-sided ideal 
of $R$ generated by $\rho (I)$.   Consider the $R$-chain complex 
$C_{\bullet}(X, R)= R\otimes_{\k} C_{\bullet}(X,\k)$, 
with differential $\tilde{\partial}^R=\id_R \otimes_{\k \pi} \tilde{\partial}$, 
and filtration $F^{\bullet}$ given by \eqref{eq:filtcxm again}. 
By Lemma~\ref{lem:irj}, the terms of this filtration can be 
expressed as
\begin{equation}
\label{eq:jfilt}
F^n C_{\bullet}(X, R)= J^n \otimes_{\k} C_{\bullet}(X,\k), 
\end{equation}
for all $n\ge 0$. Hence, by \eqref{eq:e1},
\begin{equation}
\label{eq:e1stxr}
E^1_{-s, t}(X, R)=  H_{t-s}(X, \gr^s_J (R)). 
\end{equation}

\begin{lemma}
\label{lem:dr lin}
Suppose $R$ is a commutative ring, or $R=\k{G}_{\nu}$, for 
some epimorphism $\nu\colon \pi\surj G$.  Then $E^{\bullet}(X, R)$ 
is a spectral sequence in the category of left $\gr_J(R)$-modules. 
\end{lemma}

\begin{proof}
Let $Z^r_{-s}=\{z\in F_{-s} \mid \tilde{\partial}^R z \in F_{-s-r}\}$. 
By construction, the terms of the spectral sequence are given by 
\begin{equation}
\label{eq:ers}
E^r_{-s}= Z^r_{-s}/ (Z^{r-1}_{-(s+1)} + 
\tilde{\partial}^R Z^{r-1}_{-(s+1-r)}).
\end{equation} 

From \eqref{eq:jfilt}, we find that $J^n \cdot Z^r_{-s} \subset Z^r_{-s-n}$, 
for all $n\ge 0$. Using these inclusions, together with \eqref{eq:ers}, 
we infer that 
\[
J^n \cdot E^r_{-s} \subset E^r_{-s-n}, \quad\text{for all $n\ge 0$}.  
\]
This allows us to define a natural $\gr_J(R)$-module structure on 
${}_qE^r= \bigoplus_{s\ge 0} E^r_{-s, s+q}$, for all $q\ge 0$, 
and thus on $E^r=\bigoplus_{q\ge 0}  {}_qE^r$, for all $r\ge 1$. 
(The $\gr_J(R)$-module structure on 
${}_qE^1= \bigoplus_{s\ge 0}  H_q(X, \gr_J^s(R))$ 
is induced by left-multiplication.)

Using once again \eqref{eq:ers}, it is readily checked that 
the differentials $d^r_R\colon {}_qE^r \to {}_{q-1}E^r$ act 
$\gr_J(R)$-linearly, for all $r\ge 1$. 
\end{proof}

\subsection{The differential of $E^1(X, \k{G}_{\nu})$}
\label{subsec:d1 again}

In this subsection, we assume that either $\k= \Z$ and 
$H_*(X, \Z)$ is torsion-free, or $\k$ is a field. In this case,
${}_q E^1 = \gr_J(R) \otimes_{\k} H_q(X, \k)$, by \eqref{eq:e1 new};  
in particular, ${}_q E^1$ is a free $\gr_J(R)$-module, for all $q\ge 0$. 

The above Lemma tells us that, in order to describe 
$d^1_R$ completely, it is enough to identify 
its effect on free $\gr_J (R)$-module generators. 
We will do this now, in the case when the ring 
$R$ is a group-ring, $\k{G}$, with $\k\pi$-module 
structure given by $\bar{\nu}\colon \k\pi\surj \k{G}$, the linear 
extension of an epimorphism $\nu\colon \pi\surj G$. 

As above, let $J= I_{\k} G$ be the two-sided ideal of $\k{G}$ 
generated by $\bar{\nu}(I)$, where $I=I_{\k}\pi$. Let 
\[
\nu_*\colon H_1(X,\k) \longrightarrow H_1(G,\k)
\] 
be the homomorphism induced by $\nu$ in homology with coefficients 
in $\k$. In view of Lemma \ref{lem:dr lin}, the next result describes 
the differential $d^1_G$ of $E^1(X, \k{G}_{\nu})$, solely in terms of 
$\nu_*$ and the comultiplication map $\nabla_X$ in $H_*(X,\k)$. 

\begin{prop}
\label{prop:d1 kg}
The restriction of $d^1_{G}$ to $1\otimes H_q(X, \k)$ is the 
composite
\[
\xymatrixcolsep{30pt}
\xymatrix{H_q(X, \k) \ar^(.36){\nabla_X}[r] 
& H_1(X, \k)\otimes_{\k} H_{q-1}(X, \k)
\ar^{\nu_{*} \otimes \id}[r] & H_1(G, \k)\otimes_{\k} H_{q-1}(X, \k)}.
\]
\end{prop}

\begin{proof}
Let $\mu\colon \k{G} \otimes_{\k} \k\pi \to \k{G}$ 
the multiplication map of the $\k\pi$-module $\k{G}$, given on 
group elements by $\mu(g \otimes x)= g\nu(x)$.  In view of 
Theorem \ref{thm:d1map}, we only need to identify the map 
\[
\gr (\mu) \colon \gr^0_J (\k{G})\otimes_{\k} \gr^1_I (\k\pi) 
\to \gr^1_J (\k{G}).
\]

Under the identification $\gr^0_J(\k{G})=\k$, this map 
coincides with $\gr^1 (\bar{\nu})\colon \gr^1_I (\k \pi) \to 
\gr^1_J (\k{G})$. In turn, under the identifications 
$\gr^1_I (\k\pi) = H_1(X,\k)$ and 
$\gr^1_J (\k{G}) = H_1(G,\k)$ provided by \eqref{eq:gr1}, 
the map $\gr^1 (\bar{\nu})$ coincides with $\nu_{*}$. 
This finishes the proof.
\end{proof}

The next corollary describes (under some mild hypothesis) the dual 
of the differential $d^1_G$, solely in terms of the cup-product structure 
on $H^*(X,\k)$, and the homomorphism $\nu^*\colon H^1(G,\k)\inj  
H^1(X,\k)$ induced in cohomology by the epimorphism 
$\nu\colon \pi\surj G$. 

\begin{corollary}
\label{cor:d1 transp}
Suppose $X$ is a finite-type CW-complex, and 
$H_1(G,\k)$ is a free $\k$-module.  
Let $\delta^1_G$ be the transpose of the restriction of 
$d^1_G$ to $1\otimes H_q(X, \k)$.  We then have a 
commuting triangle
\[
\xymatrixcolsep{25pt}
\xymatrix{
H^1(X, \k)\otimes_{\k} H^{q-1}(X, \k) \ar^(.6){\cup_X}[rr]&&
H^q(X, \k)\\
H^1(G, \k)\otimes_{\k} H^{q-1}(X, \k) \ar@{^{(}->}^{\nu^* \otimes \id}[u]
\ar_(.55){\delta^1_G}[urr] & 
}
\]
\end{corollary}

\subsection{Filtrations on homology}
\label{subsec:two filt}

Let $X$ be a connected CW-complex as before, $\k$ a commutative ring, and let $M$ be 
a right $\k\pi$-module. The $I$-adic filtration on $C_{\bullet}(X,M)$ 
naturally defines a descending filtration on $H_*(X,M)$.  The $n$-th 
term of the filtration is given by 
\begin{equation}
\label{eq:filthom}
F^{n} H_r(X,M) = 
\im \big(H_r(F^{n} C_{\bullet}(X, M)) \longrightarrow
H_r(X,M)\big).   
\end{equation}
Clearly, $F^0 H_r(X, M)=H_r(X, M)$. As usual, we will write 
$F_{-n}=F^n$.  The associated graded object of this filtration 
will be denoted by $\gr_F (H_*(X,M))$. 

The filtration \eqref{eq:filthom} need not be separated, 
even when the spectral sequence converges.  
We shall illustrate this phenomenon in Example \ref{ex:p3}.   

Now consider the case when $M$ is a ring $R$, with right 
$\k\pi$-module structure given by a base change $\rho\colon \k\pi\to R$. 
Let $J$ be the two-sided ideal generated by $\rho(I)$. 
As before, we will assume that either $R$ is commutative, or 
$R=\k{G}$, and $\rho=\bar{\nu}$, for some epimorphism 
$\nu\colon \pi\surj G$. 

Using \eqref{eq:jfilt}, we find that the spectral sequence filtration 
on $H_{*}(X, R)$, as defined in \eqref{eq:filthom}, is given by
\begin{equation}
\label{eq:ffilt}
F^{n}H_{*}(X, R)=\big( \ker \tilde{\partial}^R \cap 
J^n C_{\bullet}(X,R)+\im \tilde{\partial}^R \big)/\im \tilde{\partial}^R. 
\end{equation}
This puts a natural $\gr_J(R)$-module 
structure on $\gr_F H_*(X, R)$, since plainly
\begin{equation}
\label{eq:hfilts}
J^k\cdot (J^n C_{\bullet}(X,R) \cap \ker \tilde{\partial}^R)
\subset J^{k+n} C_{\bullet}(X,R) \cap \ker \tilde{\partial}^R, 
\quad \forall k, n. 
\end{equation}

\begin{remark}
\label{rem:gr einf}
Suppose $X$ is of finite type, $\k$ is a field, and 
$\dim_{\k}R/J< \infty$ (this happens automatically when 
$R=\k{G}_{\nu}$, in which case $R/J=\k$).  Fix an integer 
$q\ge 0$. Applying Proposition \ref{prop:conv ss}, and 
making use of \eqref{eq:inftyb}, we obtain an inclusion 
\begin{equation}
\label{eq:inftygr}
\gr_F H_q(X, R)\subset {}_qE^{\infty}= 
\bigoplus_{s\ge 0}E^{\infty}_{-s, s+q} ,
\end{equation}
which is  compatible with the respective $\gr_J(R)$-actions. 
\end{remark}

\subsection{Comparing the two filtrations}
\label{subsec:comp filt}

We now have two filtrations on $H_q(X, R)$:
the spectral sequence filtration $F^{\bullet}H_q(X, R)$, 
given by \eqref{eq:ffilt},
and the $J$-adic filtration, $J^{\bullet}\cdot H_q(X, R)$.
The two filtrations  are related as follows.

\begin{lemma}
\label{lem:if} 
$J^k \cdot  F^{n} H_*(X,R) \subseteq F^{k+n} H_*(X,R)$, 
for all $k, n \ge 0$. 
\end{lemma}

\begin{proof}
Our claim follows from \eqref{eq:ffilt} and \eqref{eq:hfilts}.
\end{proof}

\begin{corollary}
\label{cor:two filtrations} 
The $J$-adic filtration on $H_*(X,R)$ is finer than the 
filtration defined by $I$-adic spectral sequence:
\[
J^k \cdot H_*(X,R) \subseteq F^{k} H_*(X,R), \quad 
\text{for all $k \ge 0$}.
\] 
\end{corollary}

In general, the two filtrations differ, even in the case when 
$R=\k{A}$, with $A$ an abelian group, as the next example shows. 
On the other hand, we will see in Lemma \ref{lem:equal filt}  
that the two filtrations coincide for $A=\Z$.  

\begin{example}
\label{ex:f2}
Consider the space $X=S^1 \vee S^1$,  and its universal 
abelian cover, $X^{\ab}\to X$, classified by the map 
$\ab\colon \pi_1(X) \to \Z^2$.   The chain complex 
$C_{\bullet}(X^{\ab}, \k)$ takes the form 
$\tilde{\partial}_1 \colon \Lambda^2 \to \Lambda$, where 
$\Lambda=\k[t_1^{\pm 1}, t_2^{\pm 1}]$ and 
$\tilde{\partial}_1=\left(\begin{smallmatrix} 
t_1-1\\ t_2-1\end{smallmatrix}\right)$.
Hence,
\[
H_1(X^{\ab},\k) = \ker \tilde{\partial}_1=  
\im (t_2-1\, ,  1-t_1) \cong \Lambda.  
\]
It follows that  $J^s \cdot H_1(X^{\ab},\k)\cong J^s$.   
On the other hand, $F_{-s}H_1(X^{\ab},\k) \cong J^{s-1}$. 
Thus, the two filtrations on $H_1(X^{\ab},\k)$  
do not coincide.
\end{example}

\section{The homological Reznikov spectral sequence}
\label{sec:kzp}

In this section, we specialize to the case where the coefficient 
module is the group-ring $\k G$ of a cyclic group of prime-power 
order $p^r$, over a field $\k$ of characteristic $p$.  
When $X=K(\pi,1)$ is an Eilenberg--MacLane space 
and $G=\Z_p$, the resulting spectral sequence is the   
homological version of a cohomology spectral sequence 
described by Reznikov in \cite{Re}.

\subsection{A convergent spectral sequence}
\label{subsec:rez ss}

Let $X$ be a connected CW-complex, and let $\pi=\pi_1(X)$.  
Assume we have an epimorphism $\nu\colon \pi\surj \Z_{p^r}$. 
Fix a field $\k$ of characteristic $p$, and denote by 
$(\k\Z_{p^r})_{\nu}$ the group-ring of $\Z_{p^r}$, 
viewed as a right $\k\pi$-module via the linear extension 
$\bar{\nu}\colon \k\pi\surj \k \Z_{p^r}$. Let $I=I_\k(\pi)$ 
and $J=I_\k( \Z_{p^r} )$ be the respective augmentation ideals. 

\begin{lemma}
\label{lem:conv rez} 
The $I$-adic filtration on $C_{\bullet}(X,(\k \Z_{p^r})_{\nu})$ is finite.
\end{lemma}

\begin{proof}
By \eqref{eq:jfilt}, 
$F^n C_{\bullet}(X,(\k\Z_{p^r})_{\nu}) =J^n \otimes _{\k}  
C_{\bullet}(X,\k)$, for all $n\ge 0$. 
Now identify $\k \Z_{p^r}=\k[t]/(t^{p^r}-1)$, and 
note that $J$ is the ideal generated by $(t-1)$.  By the 
binomial formula, $(t-1)^{p^r}=t^{p^r}-1$ over $\k$.  
Thus, $J^{p^r} =0$.  Putting things together, we conclude 
that $F^{p^r} C_{\bullet}(X,(\k \Z_{p^r})_{\nu}) =0$. 
\end{proof}

Therefore, $E^{\bullet}(X,(\k\Z_{p^r})_{\nu})$ converges. 
In fact, the spectral sequence collapses in finitely many 
steps:  
$\bigoplus_{s+t=q} E^{p^r}_{s,t}(X,(\k\Z_{p^r})_{\nu})= 
H_q (X,(\k \Z_{p^r})_{\nu})$. 
For another situation when the spectral sequence 
converges, see Proposition~\ref{prop:ss kzhat} below. 

Nevertheless, the spectral sequence filtration, 
$F^{\bullet} H_*(X,(\k \Z_{p^r})_{\nu})$, does not 
always coincide with the $J$-adic filtration, 
$J^{\bullet}\cdot H_*(X,(\k \Z_{p^r})_{\nu})$, 
as the following example shows.  

\begin{example}
\label{ex:rez filt}
Let $X=S^1$, and consider the canonical projection 
$\nu\colon \pi_1(S^1)=\Z\surj \Z_p$. Identifying $\F_p \Z_p$ 
with $\Lambda_p=\F_p[t]/(t^p-1)$, the chain complex 
$C_{\bullet}(S^1,(\F_p \Z_p)_{\nu})$ takes the form 
$\tilde{\partial}_1 \colon \Lambda_p \to \Lambda_p$, 
with $\tilde{\partial}_1=(t-1)$.
It follows that $H_1(S^1, \Lambda_p)=\ker( \tilde\partial_1)$ 
is the ideal generated by the norm element, 
$N=1+\cdots +t^{p-1}$. Hence, $J \cdot H_1(S^1, \Lambda_p) =0$. 

On the other hand, $N$ augments to $0$ in $\F_p$, 
and so $N\in J\cap \ker(\tilde{\partial}_1)=F^1 H_1(S^1, \Lambda_p )$. 
Thus, $J \cdot H_1(S^1, \Lambda_p) \ne F^1 H_1(S^1, \Lambda_p )$.
\end{example}

\subsection{The Reznikov spectral sequence}
\label{subsec:rez e1}

We now specialize to the case when $G=\Z_p$ and $\k=\F_p$. 
Identify $\F_p\Z_p$ with $\F_p[t]/(t^p-1)$ and $\gr_J (\F_p \Z_p)$ 
with $\F_p [x]/(x^p)$, where $x= t-1$. 
Let  $\nu\colon \pi\surj \Z_p$ be an epimorphism, and consider 
the spectral sequence $E^{\bullet}(X,(\F_p\Z_p)_{\nu})$.  
By \eqref{eq:e1stxr}, the first term is 
\begin{equation}
\label{eq:e1kzp}
E^1_{-s,t} = \gr^s_J(\F_p\Z_p) \otimes_{\F_p}
H_{t-s}(X,\F_p) =  H_{t-s}(X,\F_p),
\end{equation}
for $t-s\ge 0$ and $0\le s\le p-1$, and $0$ otherwise.  Here 
we used that $\gr^s_J(\F_p\Z_p)=\F_p$, with basis $x^s$, 
in the specified range. 

To identify the $d^1$ differential, let $\nu_{*}\colon 
H_1(X,\F_p) \to H_1(\Z_p,\F_p)=\F_p$ be the homomorphism 
induced by $\nu$ in homology, and let $\nu_{\F_p}\in H^1(X,\F_p)$ 
be the corresponding cohomology class. 

\begin{prop}
\label{prop:d1zp}
Under the identification from \eqref{eq:e1kzp}, 
the differential $d^1\colon E^1_{-s,s+q}  \to 
E^1_{-s-1,s+q}$ is the composite 
\[
\xymatrixcolsep{28pt}
\varrho_q\colon \xymatrix{
H_q(X,\F_p) \ar[r]^(.34){\nabla_X} & H_1(X,\F_p) \otimes_{\F_p}
H_{q-1}(X,\F_p) \ar[r]^(.55){\nu_{*} \otimes \id} 
& \F_p\otimes_{\F_p}H_{q-1}(X,\F_p)
}.
\]
If, moreover, $X$ is of finite type, then the dual $\delta^1=(d^1)^{*}$ 
coincides with the left-multiplication map 
\[
\cdot \nu_{\F_p} \colon  H^{q-1}(X,\F_p) \to H^{q}(X,\F_p).
\]
\end{prop}

\begin{proof}
For the first statement, use Proposition \ref{prop:d1 kg} with 
$G= \Z_p$ and $\k=\F_p$, noting that under the identifications 
$\F_p \cong H_1(\Z_p, \F_p) \cong J/J^2 \cong \F_p \cdot x$, 
the unit $1$ corresponds to $x$. The second statement 
follows from Corollary \ref{cor:d1 transp}.
\end{proof}

In the particular case when $X=K(\pi,1)$, we recover 
the spectral sequence from \cite[Theorem~13.1]{Re}, 
in homological form. 

\begin{corollary}
\label{cor:rez ss}
Let $1\to K \to  \pi  \xrightarrow{\nu}  \Z_p \to  1$ be an exact 
sequence of groups. There is then a spectral sequence 
\[
E^1_{-s,s+q} = H_{q} (\pi,\F_p)\Rightarrow H_{q} (K, \F_p), 
\]
$q\ge 0$ and $0\le s\le p-1$, with differential 
$d^1 \colon H_q(\pi,\F_p) \to H_{q-1}(\pi,\F_p)$ 
equal to $\varrho_q=(\nu_{*} \otimes \id) \circ \nabla_{K(\pi,1)}$, 
and converging in finitely many steps.
\end{corollary}

\subsection{Monodromy action}
\label{subsec:rez mono}

In view of Proposition \ref{prop:d1zp}, the first page of 
the spectral sequence  $E^{\bullet}(X,(\F_p\Z_p)_{\nu})$ looks like 
\[
\xymatrixrowsep{8pt}
\xymatrix{
\bullet  &  \bullet \ar_{\varrho_q}[l] & \bullet 
\ar_{\varrho_{q+1}}[l]  & \bullet  \ar[l]  & \bullet  \ar[l] \\
  &  \bullet &  \bullet \ar_{\varrho_q}[l] & \bullet 
\ar_{\varrho_{q+1}}[l]  & \bullet  \ar[l]  \\
  & &   \bullet&  \bullet \ar_{\varrho_q}[l] & \bullet 
\ar_{\varrho_{q+1}}[l]  \\
  & & & \bullet&  \bullet \ar_{\varrho_q}[l] 
}
\]
with columns running right-to-left from $E^1_{0,*}$ 
to $E^1_{-p+1,*}$. Thus, if $p\ne 2$, we have 
a chain complex,
\[
\xymatrix{ \cdots \ar[r]
&   H_{q+1}(X,\F_p) \ar^(.52){\varrho_{q+1}}[r] 
&  H_{q}(X,\F_p) \ar^(.45){\varrho_{q}}[r] 
&  H_{q-1}(X,\F_p) \ar[r] & \cdots}. 
\]

\begin{prop}
\label{prop:jordan block}
If the chain complex $(H_*(X,\F_p), \varrho_*)$ is acyclic 
in degree $*=q$, then $J^2\cdot H_q(X,(\F_p \Z_p)_{\nu})=0$. 
\end{prop}

\begin{proof}
Let $F^n = F^n H_q(X,(\F_p\Z_p)_{\nu})$ be the spectral 
sequence filtration.  Our hypothesis
forces  $E^2_{-s,q+s}=0$, for $0<s<p-1$. Hence,
$E^{\infty}_{-s,q+s}=0$, for $0<s<p-1$, which implies 
$F^1 = \cdots = F^{p-1}$. 

Now recall from Lemma \ref{lem:if} that 
$J^k\cdot F^n  \subseteq F^{k+n} $.  Hence:
\[
J^2\cdot H_q(X,(\F_p \Z_p)_{\nu}) = J^2 F^0 \subseteq 
J F^1 = J F^{p-1}   \subseteq F^p =0, 
\]
and this finishes the proof.
\end{proof}

\section{Completion and homology of abelian covers}
\label{sec:completions}

In this section, we study the more general situation when the 
coefficient module is the group-ring of an abelian group $A$, 
with emphasis on the good convergence properties
of the corresponding equivariant spectral sequence. 
In the process, we use standard completion tools from 
commutative algebra, as described in \cite{Mat}  and \cite{Ser}.  

\subsection{Convergence}
\label{subsec:jadic}
Let $\k$ be a field. 
Suppose we are given an epimorphism $\nu\colon \pi\surj A$ onto an 
abelian group.  We may then view the ring $\Lambda:=\k{A}$ as a right 
$\k\pi$-module, via the ring map $\bar{\nu}\colon \k\pi\surj \k{A}$.   
We shall denote this module by $\Lambda_{\nu}$, and will 
consider the spectral sequence $E^{\bullet}(X,\Lambda_{\nu})$, 
associated to the Galois cover of the CW-complex $X$ 
corresponding to $\nu$.  

As usual, write $I=I_{\k}\pi$ and $J=\bar{\nu}(I) \Lambda$, 
and recall that $J$ is the augmentation ideal of $\Lambda=\k{A}$. 
Also, note that $J$ is a maximal ideal of the Noetherian ring $\Lambda$.  
Let $\iota \colon \Lambda \to \widehat{\Lambda}$ be the $J$-adic 
completion.  Set $\widehat{J}:= \iota (J) \widehat{\Lambda}$, and 
note that $\widehat{\Lambda}$ is a Noetherian local ring, with 
maximal ideal $\widehat{J}$, hence a Zariski ring. The ring morphism 
$\hat{\nu}:= \iota \circ \bar{\nu}$ defines another $\k\pi$-module,
which we will denote by $\widehat{\Lambda}_{\hat{\nu}}$. We
thus obtain a morphism of spectral sequences,
\[
\xymatrixcolsep{32pt}
\xymatrix{E^{\bullet} (X, \Lambda_{\nu}) 
\ar^{E^{\bullet}(\iota)}[r]& E^{\bullet} (X, \widehat{\Lambda}_{\hat{\nu}})
}.
\]

We will frequently make use of the fact that, for any finitely generated 
$\Lambda$-module $M$, there is a natural $\widehat{\Lambda}$-isomorphism,
\begin{equation}
\label{eq:jcompl}
\xymatrix{
M \otimes_{\Lambda} \widehat{\Lambda} \ar^(.6){\cong}[r]& \widehat{M}},
\end{equation}
and the completion filtration on the $J$-adic completion $\widehat{M}$ 
coincides with the $\widehat{J}$-adic filtration.

The good convergence properties of both spectral sequences follow 
easily from the Artin-Rees Lemma, as explained by Serre in \cite{Ser}.

\begin{prop}
\label{prop:ss kzhat}
Let $X$ be a connected CW-complex of finite type.
With notation as above, we have:
\begin{enumerate}
\item 
$E^1_{s,t} (X, \Lambda_{\nu})\Rightarrow H_{s+t}(X, \Lambda_{\nu})$
and $E^1_{s,t} (X, \widehat{\Lambda}_{\hat{\nu}})\Rightarrow 
H_{s+t}(X, \widehat{\Lambda}_{\hat{\nu}})$.
\item 
The filtration $F_{\bullet}H_{*}(X, \widehat{\Lambda}_{\hat{\nu}})$
is separated.
\item 
If $X$ is a finite complex, the spectral sequence 
$E^{\bullet} (X, \Lambda_{\nu})$ collapses in finitely many steps.
\end{enumerate}
\end{prop}

\begin{proof}
See \cite[pp.~22--24]{Ser}.
\end{proof}

The next lemma shows that  $E^{\bullet} (X, \Lambda_{\nu})$ 
computes $H_{*}(X, \Lambda_{\nu}) \otimes_{\Lambda} 
\widehat{\Lambda}$, the $J$-adic completion of 
$H_{*}(X, \Lambda_{\nu})$. 

\begin{lemma}
\label{lem:eiota}
Let $X$ be a connected CW-complex of finite type.
Then the maps 
\[
E^r(\iota)\colon E^r(X, \Lambda_{\nu}) \to
E^r(X, \widehat{\Lambda}_{\hat{\nu}})\quad  {\it and} \quad
\gr_F (\iota)\colon \gr_F H_{q}(X, \Lambda_{\nu}) 
\to \gr_F  H_{q}(X, \widehat{\Lambda}_{\hat{\nu}})
\]
are isomorphisms, for all $1\le r \le \infty$ and $q\ge 0$. 
\end{lemma}

\begin{proof}
Set $C_{\bullet}:= C_{\bullet}(X, \Lambda)$ and 
$\widehat{C}_{\bullet}:= C_{\bullet}(X, \widehat{\Lambda})$. 
Note that, by construction, $\widehat{C}_{\bullet}$ is the 
$J$-adic completion of the $\Lambda$-chain complex 
$C_{\bullet}$. By \eqref{eq:jfilt}, 
\[
F^k C_{\bullet}(X, \Lambda)= J^k \cdot C_{\bullet} 
\quad\text{and}\quad 
F^k C_{\bullet}(X, \widehat{\Lambda})= 
\widehat{J}^k \cdot \widehat{C}_{\bullet},
\]
for all $k$. Since $\iota \colon C_{\bullet} \to \widehat{C}_{\bullet}$ 
is the completion map, it follows from \eqref{eq:jcompl} that 
$E^0(\iota)$ is an isomorphism.  Therefore, $E^r(\iota)$ is 
an isomorphism, for  $1\le r <\infty$.  The remaining assertions 
follow from convergence.
\end{proof}

\subsection{Cyclic quotients}
\label{subsec:cycli quotient}
 
We now specialize further, to the case when $A=\Z$. 
Let $\nu\colon \pi\surj \Z$ be an epimorphism, and $\k$ a commutative 
ring. Identify the group-ring $\k{A}$ with the Laurent polynomial 
ring $\Lambda=\k[t^{\pm 1}]$, and note that the augmentation 
ideal of $\Lambda$ is principal: $J= (t-1)\Lambda$.  
Identify also the associated graded ring  $\gr(\k\Z)$ 
with the polynomial ring $\k[x]$, where 
$x=t-1$. Assuming either $\k=\Z$ and $H_*(X, \Z)$ is torsion-free, or
$\k$ is a field, it follows from \eqref{eq:e1 new} and 
Lemma \ref{lem:irj} that
\begin{equation}
\label{eq:e1kz}
E^1_{-s,t}(X,\Lambda_{\nu}) = \gr^s_J(\k\Z) \otimes_{\k} H_{t-s}(X,\k)= 
\k \otimes_{\k} H_{t-s}(X,\k) = H_{t-s}(X,\k), 
\end{equation}
where we used that $\gr^s_J(\k\Z)$ is freely generated by $x^s$. 

The induced homomorphism $\nu_{*}\colon H_1(\pi,\k)\to H_1(\Z,\k)=\k$  
determines a cohomology class $\nu_{\k}\in H^1(X,\k)$.
Proceeding as in the proof of Proposition \ref{prop:d1zp}, 
we obtain the following. 

\begin{corollary}
\label{cor:d1z}
Let $X$ be a connected CW-complex and let $\nu\colon \pi_1(X)\surj \Z$ 
be an epimorphism.  Assume either that $\k=\Z$ and $H_*(X, \Z)$ 
is torsion-free, or $\k$ is a field.
With identifications as in \eqref{eq:e1kz}, 
the differential $d^1\colon E^1_{-s,s+q}  \to E^1_{-s-1,s+q}$
is the composite 
\[
\xymatrixcolsep{30pt}
\xymatrix{
H_q(X,\k) \ar[r]^(.35){\nabla_X} & H_1(X,\k) \otimes_{\k} 
H_{q-1}(X,\k) \ar[r]^(.45){\nu_* \otimes \id} 
& \k \otimes_{\k} H_{q-1}(X,\k) = H_{q-1}(X,\k)
}\!.
\]
If, moreover, $X$ is of finite type, then the dual 
$\delta^1=(d^1)^{*}$ coincides with the left-multiplication map 
\[
\cdot \nu_{\k} \colon  H^{q-1}(X,\k) \to H^{q}(X,\k).
\]
\end{corollary}

\subsection{Comparing the two filtrations}
\label{subsec:comp filt again}

Next, identify the $J$-adic completion of $\Lambda=\k\Z$  
with the power-series ring $\widehat{\Lambda}= \k [[x]]$, 
so that the map $\iota \colon \Lambda 
\to \widehat{\Lambda}$ sends  $t$ to $x+1$.  
Note that the extension of the ideal $J$ 
is principal: $\widehat{J}= x \widehat{\Lambda}$.

Let $\nu\colon \pi \surj \Z$ be an epimorphism.
We have seen in \S\ref{subsec:exr} that each piece 
of the corresponding spectral sequences, ${}_qE^r_*$ 
and ${}_q\widehat{E}^r_*$, supports a natural 
graded module structure over the ring 
$\gr_J (\Lambda)= \gr_{\widehat{J}} (\widehat{\Lambda})$.  
Recall from Corollary \ref{cor:two filtrations} that, for any commutative 
ring $R$, viewed as a right $\k \pi$-module via a ring map 
$\rho \colon \k \pi \to R$, we have an inclusion 
$J^s \cdot H_q(X, R) \subseteq F_{-s} H_q(X, R)$, where $J$ 
is the ideal of $R$ generated by $\rho (I)$.

\begin{lemma}
\label{lem:equal filt}
Let $X$ be a connected CW-complex and $\k$ a commutative ring.
Set $R=\k\Z_\nu$ or $\widehat{\k\Z}_{\hat{\nu}}$.  Then 
\begin{enumerate}
\item \label{l1}
For each $r$ and $q$, the $\gr(R)$-module ${}_qE^r_*$ is 
generated by ${}_qE^r_0$.
\item \label{l2}
The spectral sequence filtration on $H_*(X, R)$ coincides with the
$(x)$-adic filtration.
\end{enumerate}
\end{lemma}

\begin{proof}
Part~\eqref{l1}. It is enough to check that $Z^r_{-s}= x^s\cdot Z^r_0$,
for all $r,s\ge 0$. Recall from \eqref{eq:jfilt} that 
$F^k C_{\bullet}(X, R)= x^k \cdot C_{\bullet}(X, R)$. If $z=x^s y$ and
$\tilde{\partial}(y)= x^r y'$, then $z\in F^s$ and 
$\tilde{\partial}(z)= x^{s+r} y'\in F^{s+r}$, hence $z\in Z^r_{-s}$.
Conversely, if $z=x^s y$ and 
$\tilde{\partial}(z)= x^s \tilde{\partial}(y)= x^{r+s} y'$, 
then $\tilde{\partial}(y)= x^{r} y'$, and therefore $y\in Z^r_0$, 
as needed.

Part~\eqref{l2} follows from the equality $Z^{\infty}_{-s}= 
x^s\cdot Z^{\infty}_0$, which is checked in a similar way. 
\end{proof}

Nevertheless, the filtration $F_{\bullet}(X,\k\Z)$ need 
not be separated, as the next example shows. 

\begin{example}
\label{ex:p3}
Let $X=S^1\times (S^1 \vee S^1)$, and identify  $\pi=\pi_1(X)$ 
with $\Z\times F_2=\langle a,b,c \mid a \text{ central}\rangle$.   
The epimorphism $\nu\colon \pi \surj \Z$ that sends  
$a\mapsto 2$, $b\mapsto 1$, $c\mapsto 1$ defines 
a Galois $\Z$-cover, $Y\to X$.   Identify the group ring $\k\Z$ 
with $\Lambda=\k[t^{\pm 1}]$, and consider the $\Lambda$-module 
$N:=H_1(Y,\k)=H_1(X, \k{\Z}_{\nu})$. 
Then
\[
N=\begin{cases}
\big( \Lambda/(1-t) \big)^2 \oplus \Lambda/(1+t) & \text{if $\ch\k\ne 2$},
\\[3pt]
\Lambda/(1-t)\oplus \Lambda/(1-t)^2 & \text{if $\ch\k=2$}.
\end{cases}
\]
Thus, 
\[
\bigcap_{s\ge 0} J^s N = 
\begin{cases}
\Lambda/(1+t) & \text{if $\ch\k\ne 2$}, \\
0 & \text{if $\ch\k=2$}.
\end{cases}
\]
where $J$ denotes the augmentation ideal of $\Lambda$.  
By Lemma \ref{lem:equal filt}, the filtration $\{F_{s} N\}$ is 
not separated, if $\ch\k\ne 2$.  In this situation, the spectral 
sequence $E^{\bullet} (X,\k\Z)$ converges, yet we cannot 
recover $N=H_1(X,\k\Z)$ from the spectral sequence, 
even additively. 
\end{example}

\section{Monodromy action and the Aomoto complex}
\label{sec:mono}

In the previous section, we showed that the equivariant 
spectral sequence of a finite type CW-complex $X$, 
starting at $E^{1}_{s,t} (X,\k\Z_{\nu})$, 
converges to $H_{s+t}(X,\k\Z_{\nu})$, the homology groups 
of the cover determined by an epimorphism 
$\nu\colon \pi_1(X)\surj \Z$, with coefficients in a field $\k$. 
Using this fact,  we now investigate the monodromy action 
of $\Z$ on $H_{*}(X,\k\Z_{\nu})$. In the process, we relate 
the triviality of the action to the exactness of the Aomoto 
complex defined by $\nu_{\k}\in H^1(X,\k)$. 

\subsection{Modules over $\k\Z$}
\label{subsec:kz mod}
Start by identifying the group-ring $\k\Z$ with 
$\Lambda=\k[t^{\pm 1}]$. Since $\k$ is a field, 
the ring $\Lambda$  is a PID, and so every 
finitely-generated $\Lambda$-module $P$ decomposes 
as a direct sum of a free module, $P_0=\Lambda^{\rank P}$, 
with finitely many modules of the form 
\begin{equation}
\label{eq:primary}
P_f=\bigoplus _{i> 0} (\Lambda/f^i)^{e_{f,i}(P)},
\end{equation}
indexed by irreducible polynomials $f \in \Lambda$. 
We call $P_f$ the $f$-primary part of $P$, and 
$\Lambda/f^i$ an $f$-primary Jordan block of size $i$. 

Now let $X$ be a finite-type CW complex, and 
$\nu\colon \pi_1(X)\surj \Z$ an epimorphism. For 
each $q\ge 0$, the homology group 
$P^q:=H_q(X,\k\Z_{\nu})$, viewed as a module 
over $\Lambda=\k\Z$, decomposes as
\begin{equation}
\label{eq:hdecomp}
H_q(X,\k\Z_{\nu}) = P^q_0 \oplus P^q_{t-1} \oplus 
\bigoplus_{f(1)\ne 0} P^q_f,
\end{equation}
with $P^q_0=\Lambda^{r_q}$, where $r_q$ is  
the rank of $P^q$, and $P^q_{t-1}= 
\bigoplus_{i>0} (\Lambda/(t-1)^i)^{e^q_i}$, 
where $e^q_i=e_{t-1,i}(P^q)$ is the number of 
$(t-1)$-primary Jordan blocks of size $i$. 

Our first result in this section identifies precisely the 
information that the $I$-adic spectral sequence carries 
about the $\k\Z$-module structure on  $H_*(X, \k \Z_{\nu})$, 
or, equivalently, the monodromy action of the group 
$\Z=\langle t\rangle$ on the homology groups 
of the covering space of $X$ determined by $\nu$.  
Set $J=(t-1)\k \Z$.

\begin{prop}
\label{prop:monospsq}
For each $q\ge 0$, the $\gr_J(\k \Z)$-module 
structure on ${}_qE^{\infty}(X, \k \Z_{\nu})$ described in 
\textup{\S\ref{subsec:exr}} determines the free and 
$(t-1)$-primary parts of $H_q(X, \k \Z_{\nu})$, viewed as 
a module over $\k \Z$. 
\end{prop}

\begin{proof}
Identify $\gr_J(\k \Z)= \k [x]$, where $x=t-1$. 
Proposition \ref{prop:ss kzhat}, together with 
Lemma \ref{lem:equal filt} and \eqref{eq:inftygr}, 
guarantee that the $\k [x]$-modules 
${}_qE^{\infty}$ and $\gr_J H_q(X, \k \Z_{\nu})$ 
are isomorphic. From \eqref{eq:hdecomp}, we 
obtain a decomposition 
\begin{equation}
\label{eq:kx}
\gr_J H_q(X, \k \Z_{\nu}) = \k [x]^{r_q} \oplus 
\bigoplus_{i>0} \big( \k [x]/x^i \big)^{e^q_i} ,
\end{equation}
as $\k [x]$-modules. Hence, the $\gr_J(\k \Z)$-module structure of 
${}_qE^{\infty}$ determines the numbers $r_q$ and $e^q_i$, 
that is, the free and $(t-1)$-primary parts of $H_q(X, \k \Z_{\nu})$.
\end{proof}

\subsection{Aomoto complex}
\label{subsec:aomoto}

As before, denote by $\nu_{\Z}$ the class in $H^1(X,\Z)$ 
determined by an arbitrary homomorphism $\nu\colon \pi\to \Z$.
Let $\nu_{*}\colon H_1(X,\k)\to H_1(\Z,\k)=\k$ 
be the induced homomorphism. The corresponding 
cohomology class,  $\nu_{\k}\in H^1(X,\k)$, is the image 
of $\nu_{\Z}$ under the coefficient homomorphism $\Z\to \k$. 
We then have:
\begin{equation}
\label{eq:nusquare}
\nu_{\k} \cup \nu_{\k} = 0 \text{ in } H^2(X,\k).
\end{equation}
Indeed, by obstruction theory, there is a map 
$f\colon X\to S^1$ and a class  $\omega\in H^1(S^1,\Z)$ 
such that $\nu_{\Z}=f^*(\omega)$. Hence, 
$\nu_{\Z}\cup\nu_{\Z}= f^*(\omega\cup\omega) =0$. 
Formula \eqref{eq:nusquare}  
then follows by naturality of cup products with respect to 
coefficient homomorphisms. 

As a consequence, left-multiplication by $\nu_{\k}$ turns the 
cohomology ring $H^*( X,  \k)$ into a cochain complex, 
\begin{equation}
\label{eq:aomoto}
(H^*( X,  \k) , \cdot\nu_{\k})\colon 
\xymatrix{
H^0(X,\k) \ar^{\nu_{\k}}[r] & 
H^1(X,\k) \ar^{\nu_{\k}}[r] & 
H^2(X,\k) \ar[r] & \cdots 
}
\end{equation}
which we call the  {\em Aomoto complex} of $H^*(X, \k)$, with 
respect to $\nu_{\k}$. 

\begin{definition}
\label{def:aomoto betti}
The {\em Aomoto Betti numbers} of $X$, with respect to 
the cohomology class $\nu_{\k}\in H^1(X,\k)$, are defined as 
\begin{equation}
\label{eq:aomoto betti}
\beta_q(X, \nu_{\k}) := 
\dim_{\k} H^q(H^*( X,  \k) , \cdot\nu_{\k}). 
\end{equation}
\end{definition}

Clearly, $\beta_q(X, \nu_{\k}) \le \dim_{\k} H^q( X, \k)$.  
In general though, the inequality is strict. 

\begin{example}
\label{ex:exterior}
Let $X=T^n$, the $n$-dimensional torus. The cohomology 
ring $H^*( X,  \k)$ is simply the exterior algebra on 
$H^1(X,\k)=\k^n$. If $\nu$ is onto, $\nu_{\k} \ne 0$, 
and so the complex \eqref{eq:aomoto} is exact. Thus,  
$\beta_q(X, \nu_{\k})=0$, for all $q\ge 0$, though, of 
course, $b_q(X)=\binom{n}{q}$. 
\end{example}

\subsection{Monodromy action}
\label{subsec:mono}
We are now ready to state and prove the second result 
of this section. Let $P^q=H_q(X,\k\Z_{\nu})$, viewed as a module 
over $\k\Z$ via the epimorphism $\nu\colon \pi\surj \Z$, with 
decomposition as in \eqref{eq:hdecomp}.

\begin{prop}
\label{prop:trivial action}
For each $k\ge 0$, the following are equivalent:
\begin{enumerate}
\item \label{m0}
For each $q\le k$, the vector space $H_q(X,\k\Z_{\nu})$ 
is finite-dimensional, and contains no $(t-1)$-primary 
Jordan blocks of size greater than $1$.
\item \label{m1}
For each $q\le k$, the monodromy action of $\k\Z$ on $P^q_0 \oplus 
P^{q}_{t-1}$ is trivial. 
\item \label{m2}
$\beta_0(X,\nu_{\k})=\cdots =\beta_k(X,\nu_{\k})=0$. 
\end{enumerate}
\end{prop}

\begin{proof}
The equivalence \eqref{m0} $\Leftrightarrow$ \eqref{m1} 
is an immediate consequence of decompositions 
\eqref{eq:primary}--\eqref{eq:hdecomp}.  To prove 
the equivalence \eqref{m1} $\Leftrightarrow$ \eqref{m2}, 
recall we have a spectral sequence 
$E^1_{-p,q}=\gr^p(\k\Z)\otimes H_{q-p}(X,\k)\Rightarrow 
H_q(X,\k\Z_\nu)$, whose differential $d^1\colon 
H_q(X,\k)\to H_{q-1}(X,\k)$ is the transpose of $\cdot\nu_{\k}$. 

\eqref{m2} $\Rightarrow$ \eqref{m1}. 
The condition $\beta_{\le k}(X,\nu_{\k})=0$ is equivalent 
to $E^2_{-p,q}=0$, for $q-p\le k$ and $p>0$.  
Consequently, $E^{\infty}_{-p,q}=0$, in the same range. 
By convergence of the spectral sequence and 
Lemma \ref{lem:equal filt}, this means:
\[
\frac{(t-1)^p \cdot H_{q-p}(X,\k\Z_{\nu})}
{(t-1)^{p+1} \cdot H_{q-p}(X,\k\Z_{\nu})}=0, \quad
\text{for $q-p\le k$ and $p>0$}.
\]
Taking $p=1$ in the above, and recalling the discussion 
from \S\ref{subsec:kz mod}, we obtain
\[
\frac{(t-1) \cdot (P^q_0 \oplus P^{q}_{t-1})}
{(t-1)^{2} \cdot (P^q_0 \oplus P^{q}_{t-1})}=0,
\]
for all $q\le k$. This implies $r_q=0$ (since 
$\gr^1(\Lambda)\ne 0$) and $e^{q}_{i}=0$, for $i>1$ 
(since $\dim_\k \gr^1(P^{q}_{t-1})=\sum_{i>1} e^q_i$), 
for all $q\le k$, which is equivalent to $t-1=0$ on 
$P^q_0 \oplus P^{q}_{t-1}$. 

\eqref{m1} $\Rightarrow$ \eqref{m2}.
Induction on $k$. For $k=0$, we have 
\[
\beta_0(X,\nu_{\k})= \dim_\k \ker(\cdot \nu_{\k} \colon 
H^0(X,\k)\to H^1(X,\k))=0.
\]
Now assume $\beta_{\le {k-1}}(X,\nu_{\k})=0$.  Then 
$E^2_{-1,k+1}=\cdots = E^{\infty}_{-1,k+1}$.  
Moreover, we always have $\beta_q(X,\nu_{\k})=\dim_\k 
E^2_{-p,q+p}$, for all $p\ge 1$.  In our situation, we have 
\[
\beta_k(X,\nu_{\k})= \dim_\k E^{\infty}_{-1,k+1}
= 
\dim_\k\frac{(t-1) \cdot (P^k_0 \oplus P^{k}_{t-1})}
{(t-1)^{2} \cdot (P^k_0 \oplus P^{k}_{t-1})},
\]
and this vanishes, by assumption.  The induction step 
is thus proved.
\end{proof}

\section{Bounds on twisted cohomology ranks: I}
\label{sect:betti bounds}

In this section, we give  upper bounds on the 
ranks of the cohomology groups with coefficients in 
a rank $1$ local system defined by a rational character 
of prime-power order.  

\subsection{Twisted Betti numbers}
\label{subs:twist betti}
We start with some definitions.  Let $X$ be a connected, 
finite-type CW-complex.  Given a field $\k$, the 
{\em $\k$-Betti numbers} of $X$ are defined as  
$b_q(X, \k) := \dim _{\k} H^q( X,  \k)$. 
More generally, if $R$ is a Noetherian ring, define 
$b_q(X, R)$ to be the minimal number of generators 
of the $R$-module $H^q(X,R)$. 

Now suppose $\rho \colon \pi\to \C^{\times}$ is a homomorphism 
from  $\pi=\pi_1(X)$ to the multiplicative group of non-zero 
complex numbers. The {\em twisted Betti numbers} of $X$ 
corresponding to the character $\rho$ are defined by
\begin{equation}
\label{eq:twisted betti}
b_q(X, \rho) := \dim _{\C} H^q( X,  {}_{\rho}\C), 
\end{equation}
where ${}_{\rho}\C=\C$, viewed as a left module over $\C\pi$, via 
$\rho$.  By duality, $b_q(X, \rho) = \dim _{\C} H_q( X, \C_{\rho})$, 
where $\C_{\rho}=\C$, viewed as a right module over
$\C\pi$, via $\rho$.  

Of particular importance to us are characters of the form 
$\rho(g)=\zeta^{\nu(g)}$, where $\nu\colon \pi \to \Z$ is a 
homomorphism, $d$ is a positive integer, and $\zeta$ is 
a primitive $d$-th root of unity.  In this case, we say $\rho$ 
is a {\em rational character} of order $d$, and write 
\begin{equation}
\label{eq:nud betti}
b_q(X, \nu/d) := b_q(X,\rho). 
\end{equation}
It follows from elementary Galois 
theory that the twisted Betti numbers $b_q(X, \nu/d)$ do 
not depend on the choice of primitive $d$-th root of unity, 
thereby justifying the notation. 

\subsection{Cyclotomic polynomials}
\label{subs:cyclo}
For each positive integer $d$, the $d$-th cyclotomic polynomial 
is defined as $\Phi_d(t)=\prod_{\zeta} (t-\zeta)$, where 
$\zeta$ ranges over all primitive $d$-th roots of unity. 
If $d=p^r$, with $p$ a prime, then 
$\Phi_{d}(t)=(t^{p^r}-1)/(t^{p^{r-1}}-1)$, and so 
$\Phi_{d}(1)=p$. If $d$ is not a prime power, and $d>1$, 
then $\Phi_{d}(1)=1$. 
 
\begin{lemma}
\label{lem:cyclo}
Let $Q\in \Z[t]$ be a polynomial with integer coefficients.  
Suppose $Q(\zeta)=0$, for some root of unity $\zeta\in \C$ 
of prime-power order $p^r$. Then $Q(1)=0  \pmod{p}$. 
\end{lemma}

\begin{proof}
The minimal polynomial of $\zeta$ is the cyclotomic polynomial 
$\Phi_{p^r}(t)$. Thus, $\Phi_{p^r}\mid Q$. But  
$\Phi_{p^r}(1)=p$, and so $p\mid Q(1)$. 
\end{proof}

The following Corollary will be useful in the sequel. 
Let $A$ be a matrix with entries in $\Z[t^{\pm 1}]$.  
We will denote by $A(z)$ the evaluation 
of $A$ at a non-zero complex number $z\in \C^{\times}$.  
For $p$ a prime, we may also view $A(1)$, after 
reduction modulo~$p$, as a matrix with entries in $\F_p$. 

\begin{corollary} 
\label{cor:minors}
Let $\zeta$ be a root of unity of order a power of a 
prime $p$. Then 
\[
\rank_\C A(\zeta) \ge \rank_{\F_p} A(1).
\]
\end{corollary}

\begin{proof}
Suppose $m(t)$ is a minor of $A$, and $m(\zeta)=0$. 
Then $t^k m(t)\in \Z[t]$, for some $k\ge 0$, and 
$\zeta^k m(\zeta)=0$. Hence, by Lemma \ref{lem:cyclo}, 
$m(1)=0$ in $\F_p$. The conclusion follows. 
\end{proof}

\subsection{Modular Betti bounds}
\label{subs:mod bb bound}
The following result relates the two kinds of Betti numbers 
defined above, under a prime-power assumption on 
the order of the rational character.  
\begin{theorem}
\label{thm:bettibound}
Let $X$ be a connected, finite-type CW-complex.
Let  $\nu\colon \pi_1(X)\to \Z$ be a homomorphism, 
and $p$ a prime. Then, for all $r\ge 1$ and $q\ge 0$, 
\begin{equation}
\label{eq:bettibound}
b_q(X,\nu/p^r) \le b_q(X,\F_p).
\end{equation}
\end{theorem}

\begin{proof}
Let $C_{\bullet}(X,\Z)=(C_q,\partial_q)_{q\ge 0}$ be the 
cellular chain complex of $X$, and let $\Z\Z_{\nu}$ 
be the ring $\Z\Z=\Z[t^{\pm 1}]$, viewed as a module over 
$\Z\pi_1(X)$ via the ring map $\bar{\nu}\colon \Z\pi_1(X) \to \Z\Z$. 
The equivariant chain complex $C_{\bullet}(X,\Z\Z_{\nu})$  
has chains $C_q(X,\Z\Z_\nu) = \Z[t^{\pm 1}] \otimes  C_q$, 
and  differentials 
\begin{equation}
\label{eq:del nu}
\dnu_q:=\tilde{\partial}^{\Z\Z_{\nu}}_q\colon 
\Z[t^{\pm 1}] \otimes  C_q\to\Z[t^{\pm 1}] \otimes  C_{q-1}.
\end{equation}  
Note that $\dnu_q(1)=\partial_q$. 
From the definition of twisted Betti numbers, we have
\begin{equation}
\label{eq:bettitwist}
b_q(X, \nu/p^r) = \rank C_q -\rank_\C \dnu_q(\zeta)
  -\rank_\C \dnu_{q+1}(\zeta), 
\end{equation}
where $\zeta$ is a primitive root of 
unity of order $p^r$.    

Now consider the chain complex $C_{\bullet}(X,\F_p)$. 
By definition, $C_q(X,\F_p)=C_q\otimes \F_p$.  Clearly, 
the differential $\partial_q \otimes \id_{\F_p}$ equals 
the reduction mod $p$ of $\dnu_q(1)$. Thus,
\begin{equation}
\label{eq:betti p}
b_q(X,\F_p) = \rank C_q -\rank_{\F_p} \dnu_q(1)
  -\rank_{\F_p} \dnu_{q+1}(1).
\end{equation}

The desired inequality follows from \eqref{eq:bettitwist}, 
\eqref{eq:betti p}, and Corollary \ref{cor:minors}.  
\end{proof}

\begin{remark}
\label{rem:knots}
Given a knot $K$ in $S^3$, let $X=S^3 \setminus K$ be 
the knot complement, $\pi=\pi_1(X)$ the knot group, and 
$\nu=\ab\colon \pi \surj \Z$ the abelianization map.   
The {\em Alexander polynomial} of the knot, 
$\Delta_K(t)\in \Z[t^{\pm 1}]$, is the greatest common 
divisor of the codimension $1$ minors of the Alexander 
matrix, $\partial_2^{\nu}$, defined as in \eqref{eq:del nu}.  
Let $\rho \colon \pi\to \C^{\times}$ be a non-trivial character.  
Writing $\rho(g)=z^{\nu(g)}$, for some $z\in  \C^{\times}$, 
we have:
\begin{equation}
\label{eq:alex root}
b_1(X,\rho)\ne 0  \same \Delta_K(z)=0
\end{equation}
(see \cite{DPS07} for a much more general statement). 
In particular, $b_1(X,\nu/d)\ne 0$ if and only if $\Delta_K(\zeta)=0$, 
for some primitive $d$-th root of unity $\zeta$. 
\end{remark}

\subsection{Examples and discussion}
\label{subsec:sharp bettibound} 
We conclude this section by discussing the necessity 
of the hypothesis in Theorem \ref{thm:bettibound}, 
and the sharpness of inequality \eqref{eq:bettibound}.

We start with the prime-power hypothesis. Suppose $d$ is a 
positive integer so that $d\ne p^r$, for any prime $p$. 
One may wonder whether an inequality of the form
\begin{equation}
\label{eq:betti general}
b_q(X,\nu/d) \le b_q(X,R)
\end{equation}
holds, for some suitable choice of Noetherian ring $R$.  
The following example shows that this is not possible, 
in general.  

\begin{example}
\label{ex:knots1}
By assumption, $d$ is not a prime-power integer;  
thus, $\Phi_d(1)=1$.  On the other hand, 
$\Phi_d(t^{-1})\equiv \Phi_d(t)$, up to units in $\Z[t^{\pm 1}]$.  
Hence, by a classical result of Seifert \cite{Se}, there 
is a knot $K_d$ in $S^3$ with Alexander polynomial 
$\Delta_{K_d}(t)=\Phi_d(t)$.  

Fix an integer $n>1$, and let $K=\sharp^n K_d$ be the 
connected sum of  
$n$ copies of $K_d$. Denote by $X_d =S^3 \setminus K_d$ 
and $X =S^3 \setminus K$ the corresponding knot complements. 
By additivity of Alexander invariants under connected sums of 
knots (see for instance \cite{Ro}), we have an isomorphism 
of $\C\Z$-modules, 
$H_1(X, \C \Z_{\nu})\cong  H_1(X_d, \C \Z_{\nu_d})^n$, 
where the twisted coefficients are defined by abelianization.
Let $\C_{\zeta}=\C$, viewed as a $\C\Z$-module via evaluation of
Laurent polynomials at a primitive $d$-root of unity $\zeta$. 
Then
\begin{align*}
b_1(X, \nu/d)&= \dim_{\C} (\C_{\zeta} \otimes_{\C\Z} 
H_1(X, \C\Z_{\nu})) \\ 
&= n \dim_{\C} (\C_{\zeta} \otimes_{\C\Z} H_1(X_d, \C\Z_{\nu_d})) \\
& = n\, b_1(X_d, \nu_d/d). 
\end{align*}
But $\Delta_{K_d}(\zeta)=\Phi_d(\zeta)=0$, and so, as noted in 
Remark \ref{rem:knots}, $b_1(X_d, \nu_d/d)\ge 1$. 
Thus, $b_1(X, \nu/d)\ge n$.  On the other hand, 
$b_1(X,R)=1$, for any ring $R$.    
\end{example}

Next, we show that inequality \eqref{eq:bettibound} 
from Theorem \ref{thm:bettibound} may fail to 
be an equality.

\begin{example}
\label{ex:knots2}

Let $X=S^3\setminus K$ be a knot complement, 
$\nu\colon \pi_1(X)\surj \Z$ the abelianization map, 
and $d=p^r$.  If $b_1(X,\nu/d)$ were non-zero,  then 
$\Delta_K(\zeta)$ would vanish, and so $\Phi_d(t)$ would divide 
$\Delta_K(t)$.  But $\Phi_d(1)=p$, while $\Delta_K(1)=\pm 1$. 
Hence, we have $b_1(X,\nu/d)=0$, yet $b_1(X,\F_p)=1$.   
\end{example}

Finally, we show that the bound  \eqref{eq:bettibound} 
cannot be improved.

\begin{example}
\label{ex:sharp}
Let $X$ be any connected, finite-type  CW-complex with 
$H_q(X,\Z)$ torsion-free, and let $\nu\colon \pi_1(X)\to \Z$ 
be the zero map.  Then $b_q(X,\nu/p^r)=b_q(X,\F_p)=b_q(X)$, 
for all primes $p$.  
\end{example}

\section{Bounds on twisted cohomology ranks: II}
\label{sect:aomoto bounds}

Continuing the theme from the previous section, 
we sharpen the bounds on the twisted Betti numbers 
$b_q(X,\nu/p^r)$, by using the Aomoto Betti numbers 
$\beta_q(X,\nu_{\F_p})$ instead of the 
usual Betti numbers $b_q(X,\F_p)$, under an 
additional freeness assumption on $H_*(X,\Z)$. 

\subsection{Decomposing the boundary map}
\label{subsec:homology}
Let $X$ be a connected, finite-type CW-complex.
Assume $\k=\Z$ and $H_*(X,\Z)$ is torsion-free, or 
$\k$ is a field.  Let $C_q=C_q(X,\k)$ be the the group of 
cellular $q$-chains over $\k$, and $\partial_q\colon C_q \to C_{q-1}$ 
the boundary map. Writing $Z_q=\ker \partial_q$ and 
$B_{q-1}=\im \partial_{q}$, we have a split exact sequence 
\[
0\to Z_q \to C_q \xrightarrow{\partial_q} B_{q-1}\to 0.
\] 
By assumption, $H_q=Z_q/B_q$ is a free $\k$-module.  
Thus, 
\begin{equation}
\label{eq:chain decomp}
C_q \cong Z_q \oplus B_{q-1} \cong B_q \oplus N_q,
\end{equation}
where $N_q \cong H_q\oplus B_{q-1}$.  The next Lemma follows at once. 

\begin{lemma}
\label{eq:block}
With respect to the direct sum decompositions $C_q=Z_q \oplus B_{q-1}$ 
and $C_{q-1}=B_{q-1} \oplus N_{q-1}$, the matrix of the differential 
$\partial_q\colon C_q \to C_{q-1}$ takes the block-matrix form 
\begin{equation}
\label{eq:del decomp}
\partial_q =\begin{pmatrix}
0  & \id \\
0 & 0
\end{pmatrix}.
\end{equation} 
\end{lemma}

\begin{remark}
\label{rem:uct}
Suppose $H_*(X,\Z)$ is torsion-free. Then, by standard 
homological algebra, $Z_q(X,\k)=Z_q(X,\Z)\otimes \k$, 
$B_q(X,\k)=B_q(X,\Z)\otimes \k$, and 
$H_q(X,\k)=H_q(X,\Z)\otimes \k$. It easily follows that 
the block-matrix decomposition \eqref{eq:del decomp}  
of $\partial_q$  is compatible with the canonical 
coefficient homomorphism $\Z\to \k$. 
\end{remark}

\subsection{Aomoto bounds on the twisted Betti numbers}
\label{subsec:aomoto bound}
We are now ready to state the main result of this section.  

\begin{theorem}
\label{thm:cohobound}
Let $X$ be a connected, finite-type CW-complex, with 
$H_*(X,\Z)$ torsion-free. Let  $\nu\colon \pi_1(X)\to \Z$ be 
a homomorphism, and $p$ a prime.  Then, 
for all $r\ge 1$ and $q\ge 0$, 
\begin{equation}
\label{eq:cohobound}
b_q(X, \nu/p^r) \le \beta_q (X,\nu_{\F_p}).
\end{equation}
\end{theorem}

\begin{proof}
Without loss of generality, we may assume $\nu$ is surjective. 
Indeed, if $\nu =0$ then clearly 
\[
b_q(X, \nu/p^r) = b_q(X, \C) \le 
b_q(X, \F_p) =\beta_q (X,\nu_{\F_p}).
\]
On the other hand, if the image of $\nu$ has index $m$ in $\Z$, 
then plainly $b_q(X, \nu/p^r) = b_q(X, \nu'/p^s)$, for an epimorphism 
$\nu'\colon \pi_1(X)\surj \Z$ such that $\nu =m \nu'$. 
(Here $p^s$ is the order of $\zeta^m$, where $\zeta$ is 
a primitive root of unity of order $p^r$.) If $p$ divides 
$m$, then $\nu_{\F_p}= 0$ and claim \eqref{eq:cohobound} becomes 
inequality \eqref{eq:bettibound}. Otherwise, $s\ge 1$ and 
$\beta_q (X,\nu_{\F_p})= \beta_q (X,\nu'_{\F_p})$. Hence, 
the result for $\nu'$ implies the claim for $\nu$.

Consider the equivariant chain complex 
$C_{\bullet}(X,\Z\Z_\nu)$, viewed as a chain 
complex over $\Lambda=\Z[t^{\pm 1}]$.  Identify 
$C_q(X,\Z\Z_\nu) = \Lambda \otimes C_q$, 
where $C_q=C_q(X,\Z)$, and  denote by  
$\dnu_q\colon \Lambda\otimes C_q\to 
\Lambda \otimes C_{q-1}$ the boundary maps.  

Since $\dnu_q(1)=\partial_q$ and $\left. \partial_q \right|_{Z_q}=0$, 
the restriction of $\dnu_q$ to $\Lambda\otimes Z_q$ takes values in 
$J\otimes C_{q-1}$, where $J=(t-1)\Lambda$.  
As in \S\ref{subsec:homology}, write $C_q= Z_q \oplus B_{q-1}$ 
and $C_{q-1} = N_{q-1}\oplus B_{q-1}$. Using 
formula \eqref{eq:del decomp}, we see that 
$\dnu_q$ takes the block-matrix form 
\begin{equation}
\label{eq:dq}
\dnu_q =\begin{pmatrix}
(t-1) \Ps  & \Qs \\
(t-1) \Rs & \Ss
\end{pmatrix},
\end{equation}
where $\Qs(1)$ is the identity, and $\Ss(1)$ 
is the zero matrix. For future use, define the block-matrix 
$\As_q:=\left(\begin{smallmatrix} \Ps  & \Qs \\ \Rs & \Ss
\end{smallmatrix}\right)$, 
and note that $\rank_\C \dnu_q(z) = \rank_\C \As_q(z)$, 
for any $z\in \C\setminus \{0,1\}$.  

Recall from \eqref{eq:bettitwist} that 
$b_q(X, \nu/p^r) = \rank C_q -\rank_\C \dnu_q(\zeta)
  -\rank_\C \dnu_{q+1}(\zeta)$, 
where  $\zeta$ is a primitive root of 
unity of order $p^r$.   Hence, 
\begin{equation}
\label{eq:bqa}
b_q(X, \nu/p^r) =\rank C_q -\rank_\C \As_q(\zeta) 
  -\rank_\C \As_{q+1}(\zeta).
\end{equation}

Let us estimate the right side.  By Corollary \ref{cor:minors}, 
\begin{equation}
\label{eq:aqz}
\rank_\C \As_q(\zeta) \ge \rank_{\F_p} \As_q(1).
\end{equation}
By Lemma \ref{eq:block} and Remark \ref{rem:uct},  
\begin{equation}
\label{eq:arb}
\rank_{\F_p} \As_q(1)=\rank_{\F_p} \Rs_q(1) 
+ \rank B_{q-1}. 
\end{equation}
Combining \eqref{eq:bqa},  \eqref{eq:aqz}, and \eqref{eq:arb} with 
the equality 
$b_q(X)= \rank C_q -\rank  B_q -\rank  B_{q-1}$, we obtain
\begin{equation}
\label{eq:bqnupr}
b_q(X, \nu/p^r)
\le  b_q(X)  -\rank_{\F_p}  \Rs_q(1)  -\rank_{\F_p}  \Rs_{q+1}(1).
\end{equation}

We now turn to the Aomoto-Betti numbers. By definition, 
\begin{equation}
\label{eq:betaqnufp}
\beta_q(X,\nu_{\F_p}) = b_q(X,\F_p)-\rank_{\F_p} \nu^q_{\F_p}
  -\rank_{\F_p} \nu^{q+1}_{\F_p}, 
\end{equation}
where $\nu^q_{\F_p} \colon H^{q-1}(X,\F_p)\to 
H^{q}(X,\F_p)$ denotes left-multiplication by 
$\nu_{\F_p}\in H^1(X,\F_p)$.  
By Corollary \ref{cor:d1z}, this map is dual to 
the differential $d^1_q\colon E^1_{0,q} \to E^1_{-1,q}$, 
where $E^1_{0,q} = \gr^0(\F_p\Z)\otimes_{\F_p} H_q(X,\F_p)$ 
and $E^1_{-1,q} = \gr^1(\F_p\Z)\otimes_{\F_p} H_{q-1}(X,\F_p)$.   
Using the identifications $\gr^0(\F_p\Z)=\F_p$ and 
$\gr^1(\F_p\Z)=x\cdot \F_p\cong \F_p$, where 
$x=t-1$, we may view this differential as a map 
$d^1_q\colon H_q(X,\F_p) \to H_{q-1}(X,\F_p)$. 

Clearly, $d^1_q$ has the same rank as 
$\delta^1_q=\iota \circ d^1_q \circ \pi$, 
where $\pi\colon Z_q(X,\F_p) \surj H_q(X,\F_p)$ is the 
projection, and $\iota \colon H_{q-1}(X,\F_p) \inj 
N_{q-1}(X,\F_p)$ is the inclusion. 
Hence, 
\begin{equation}
\label{eq:betaqd1}
\beta_q(X,\nu_{\F_p}) = b_q(X)-\rank_{\F_p} \delta^1_q
  -\rank_{\F_p} \delta^1_{q+1}.
\end{equation}

In view of \eqref{eq:bqnupr} and \eqref{eq:betaqd1}, 
it suffices to show  that $\delta^1_q=\Rs_q(1)$.  
Let $z\in Z_q(X,\F_p)$.  Using formula \eqref{eq:dq}---%
with everything reduced mod $p$---%
we find:
\begin{align*}
\delta^1_q(z) 
& = \iota( [\dnu_q(1 \otimes z) \bmod J^2] ) \\
&= (t-1) \Rs_q(t) z \bmod J^2 \\
&=x \Rs_q(1) z \bmod J^2\\
&\equiv \Rs_q(1) z
\end{align*}
where at the last step we used the identification 
$\gr^1(\F_p\Z)=J/J^2=x\cdot \F_p\cong \F_p$. 
This finishes the proof.
\end{proof}

\subsection{Necessity of the hypothesis}
\label{subsec:hypo}
We now give examples showing that the two hypotheses 
in Theorem \ref{thm:cohobound} are necessary.  We start 
with the prime-power hypothesis. 

Given a Noetherian ring $R$, let $\nu_R\in H^1(X,R)$ be the 
cohomology class determined by the homomorphism 
$\pi_1(X)\xrightarrow{\nu} \Z \xrightarrow{\iota} R$, 
where $\iota(1)=1$.  Define $\beta_q(X, \nu_R)$ to 
be the minimal number of generators of the $R$-module 
$H^q(H^*( X,R) , \cdot\nu_R)$.  

Now suppose $d$ is a 
positive integer so that $d\ne p^r$, for any prime $p$. 
One may wonder whether an inequality of the form
$b_q(X,\nu/d) \le \beta_q(X,\nu_R)$ holds, for some 
suitable choice of Noetherian ring $R$.  
The following example shows that this is not possible. 

\begin{example}
\label{ex:coho torus}
Let us start by recalling an old result of R.~Lyndon  
(see \cite[Thm.~3.6]{Bi}). If $v_1,\dots,v_n$ are  
elements in $\Lambda= \Z\Z^n$ satisfying 
$\sum_{i=1}^n v_i(t_i-1)=0$, then there is a 
word $r\in F'_n$ such that $(\partial r/\partial x_i)^{\ab}=v_i$, 
for all $i$. Hence, if $\pi=\langle x_1,\dots, x_n\mid r\rangle$ 
is the corresponding $1$-relator group,  and $X$ is the 
presentation $2$-complex, then the chain 
complex $C_{\bullet}(X^{\ab},\k)$ has boundary 
maps $\widetilde{\partial}^{\ab}_2\colon \Lambda \to \Lambda^n$ 
and $\widetilde{\partial}^{\ab}_1\colon \Lambda^n \to \Lambda$ 
given by
\[
\widetilde{\partial}^{\ab}_2(1 \otimes e_2) 
=\sum_{i=1}^n v_i e^i_1
\quad\text{and}\quad 
\widetilde{\partial}^{\ab}_1(e^i_1) = 
(t_i -1) e_0.
\] 

Now take  $v_1=\Phi_d(t_1) (t_2-1)$ and 
$v_2=\Phi_d(t_1) (1-t_1)$, and let 
$X$ be the $2$-complex constructed above.   
Let $\nu\colon \pi\surj \Z$, $\nu(x_1)=\nu(x_2)=1$, 
and fix a primitive $d$-th root of unity $\zeta$. The 
chain complex $C_{\bullet}(X,\C_{\rho})\colon \C \to \C^2 \to \C$ 
corresponding to the character  $\rho\colon \pi\to \C^{\times}$, 
$\rho(g)=\zeta^{\nu(g)}$, has boundary maps 
$\partial^{\rho}_2=0$ and 
$\partial^{\rho}_1 =\Big(\begin{smallmatrix}\zeta-1\\[2pt] 
\zeta-1\end{smallmatrix}\Big)$. 
Hence, $H_1(X,\C_{\rho})=\C$, and so $b_1(X,\nu/d)=1$.  

Next, let $R$ be a Noetherian ring, and set $J=I_R\Z^2$.  
The differential  $d^1\colon H_2(X,R) \to H_1(X,R)\otimes H_1(X,R)$ 
on $E^1(X,R\Z^2_{\ab})$ is given by:
\begin{align*}
d^1( \brac{e_2}) 
&= \widetilde{\partial}^{\ab}_2(1\otimes e_2)  \mod J^2 \\
&=\Phi_d(t_1) (t_2-1) \otimes \brac{e^1_{1}}  + 
\Phi_d(t_1) (1-t_1) \otimes  \brac{e^2_{1}}  \mod J^2 \\
&=\Phi_d(1) (t_2-1) \otimes  \brac{e^1_{1}}  + 
\Phi_d(1) (1-t_1) \otimes  \brac{e^2_{1}}  \mod J^2 \\ 
&\equiv \Phi_d(1) \left( \brac{e^2_{1}} \otimes  \brac{e^1_{1}} - 
 \brac{e^1_{1}} \otimes  \brac{e^2_{1}}\right). 
\end{align*}
Hence, by Corollary \ref{cor:d1z}, 
the map $\cdot \nu_R \colon H^1(X,R) \to H^2(X,R)$ 
sends $\brac{e^{1}_{1}}^*\mapsto \Phi_d(1)\brac{e_2}^*$ 
and $\brac{e^{2}_{1}}^*\mapsto -\Phi_d(1)\brac{e_2}^*$. 
On the other hand, the map 
$\cdot \nu_R \colon H^0(X, R) \to H^1(X, R)$ 
sends $\brac{e_{0}}^*\mapsto \brac{e^1_1}^*+\brac{e^2_1}^*$.
Recall we assumed $d$ is not a prime power, i.e., $\Phi_d(1)=1$. 
Therefore, $H^1(H^*(X,R),\nu_R) = 0$.

To recap, we showed that $b_1(X,\nu/d)=1$, yet $\beta_1(X,\nu_R)=0$. 
Let us note that $X$ is a minimal CW-complex, in the sense 
of the definition from \S\ref{subsect:mini} below; indeed, 
$\epsilon(v_1)=\epsilon(v_2)=0$, and so the boundary maps in 
$C_{\bullet}(X,\Z)$ vanish.  
Furthermore,  Corollary \ref{cor:d1 transp} implies that 
$X$ has the same cohomology ring as $S^1\times S^1$. 
\end{example}

Next, we show that the hypothesis that $H_*(X,\Z)$ 
be torsion-free is really necessary.

\begin{example}
\label{ex:tors free}
Start with $Y=S^1 \vee S^2$, and identify $\pi_1(Y)=\Z$, 
with generator $t$, and set $\Lambda=\Z[t^{\pm 1}]$. Using 
the construction from Example \ref{ex:construction}, build the  
CW-complex $X=Y\cup_{\phi_{1+t}} e_3$.  The equivariant 
chain complex of $X$ can be written as 
\begin{equation}
\label{eq:equiv2}
C_{\bullet}(\wX,\Z)\colon
\xymatrix{\Lambda \ar[r]^{1+t} & \Lambda \ar[r]^{0} 
& \Lambda \ar[r]^{t-1} &\Lambda}.
\end{equation}
Consequently, $C_{\bullet}(X,\Z)$ has the form 
$\Z \xrightarrow{2} \Z \xrightarrow{0} \Z 
\xrightarrow{0} \Z$, and so $H_2(X,\Z)=\Z_2$.  

Now take $\nu\colon \pi_1(X)\to \Z$ to be the identity,  
and pick the prime $p=2$.  The chain complex 
$C_{\bullet}(X,\C_{\rho})$ corresponding to the 
rational character $\rho(t)=-1$ has the form 
$\C \xrightarrow{0} \C \xrightarrow{0} 
\C \xrightarrow{-2} \C$, and so $b_3(X,\nu/2)=1$.

On the other hand,  it follows from \eqref{eq:equiv2}
that all boundary maps of $C_{\bullet}(X,\F_2)$ 
are $0$. Hence, $H_i(X,\F_2)=\F_2$, 
generated by $\brac{e_i}$, for $0\le i\le 3$. Moreover, 
$\nu_{\F_2}=\brac{e_1}^*$, the generator of 
$H^1(X,\F_2)=\F_2$. 
We also know from \eqref{eq:equiv2} that 
$\tilde{\partial}_3(1\otimes e_3)= (1+t) \otimes e_2$. Hence, the differential 
$d^1\colon H_3(X,\F_2)\to  H_2(X,\F_2)$ of $E^1 (X, \F_2 \Z_{\nu})$
is given by $\brac{e_3} \mapsto  \brac{e_2}$.  
By Corollary \ref{cor:d1z}, the map 
$\cdot \nu_{\F_2} \colon H^2(X,\F_2) \to H^3(X,\F_2)$ 
takes $\brac{e_2}^*$ to $\brac{e_3}^*$. Hence, 
$\beta_3(X,\nu_{\F_2})=0$. 
\end{example}

\subsection{Sharpness of the bound}
\label{subsec:sharp}
Inequality \eqref{eq:cohobound} from 
Theorem \ref{thm:cohobound}  may fail to 
be an equality, as we now show.

\begin{example}
\label{ex:coho strict}

Let $G=\langle x, y \mid [x,y]^p  \rangle$, let $X$ be 
the associated presentation $2$-complex, and 
let $\nu\colon G\to \Z$ be the diagonal character, 
sending both $x$ and $y$ to $1$. We then have 
$b_1(X,\nu/p)=0$, whereas $\beta_1(X,\nu_{\F_p})=1$.
\end{example}

On the other hand, the bound  \eqref{eq:cohobound} 
cannot be improved, as the next example shows.  

\begin{example}
\label{ex:coho sharp}

Let $\A$ be a subarrangement of a complexified reflection 
arrangement of type $A$, $B$, or $D$, and let $X$ be 
the complement of $\A$.  
Choose $\nu\colon \pi_1(X) \to \Z$ to be the diagonal 
character, sending each oriented meridian to $1$.  
In this case, the bound in Theorem \ref{thm:cohobound} 
is attained at all primes, for $q=1$ and $r=1$; see 
\cite[Theorem C]{MP}. 
\end{example}

\section{Galois covers, minimality, and linearization}
\label{sect:minilin}

In this section, we analyze in detail the first page of the
equivariant spectral sequence of an arbitrary Galois cover. 
Using this approach, we give an intrinsic meaning to the 
linearization of the equivariant chain complex, in the 
important case of minimal CW-complexes.  Throughout, 
$\k$ will denote the integers $\Z$, or a field. 

\subsection{Minimal cell complexes}
\label{subsect:mini}

We start by reviewing a notion discussed in \cite{PS};  
see also \cite{DP1, DP2} for various applications.
Let $X$ be a connected, finite-type CW-complex.  
We say the CW-structure on $X$ is {\em minimal} 
if the number of $q$-cells of $X$ coincides with the 
(rational) Betti number $b_q(X)$, for every $q\ge 0$. 
Equivalently, the boundary maps in the cellular 
chain complex $C_{\bullet}(X,\Z)$ are all the zero maps. 
In particular, $X$ has a single $0$-cell, call it $e_0$.  

If $X$ is a minimal cell complex, the homology groups 
$H_*(X,\Z)$ are all torsion-free. In particular, 
$H_q(X,\k)=H_q(X,\Z)\otimes \k$ and 
$H^q(X,\k)=H_q(X,\k)^*$, for $\k=\Z$ or $\k$ a field.

Even if the homology groups of $X$ are torsion-free, the 
space $X$ need not be minimal.  For example,  if $K$ is 
a knot in $S^3$, with complement $X$, then 
$H_*(X,\Z)=H_*(S^1,\Z)$, yet $X$ does not admit a 
minimal CW-structure, unless $K$ is the trivial knot.

Examples of spaces admitting minimal CW-structures  
are: spheres $S^n$, tori $T^n$, orientable Riemann surfaces, 
and complex Grassmanians---in fact, any compact, connected 
smooth manifold admitting a perfect Morse function.  If $\A$ 
is a complex hyperplane arrangement, then its complement, 
$X$, has a minimal cell decomposition; see \cite{DP1}. 

\subsection{Linearizing the equivariant boundary maps}
\label{subsec:lin diff}
Let $X$ be a minimal CW-complex, with 
fundamental group $\pi=\pi_1(X)$. As usual, let 
$(C_{\bullet}(X,\k), \partial)$ be the cellular chain 
complex of $X$, and let $(C_{\bullet}(\wX,\k), \tilde{\partial})$ 
be the equivariant chain complex, with filtration 
$F^n= I^n \cdot C_{\bullet} (\wX, \k)$, where $I=I_{\k} \pi$.    
Let  $\nu\colon \pi \surj G$ be an epimorphism, and consider 
the chain complex $(C_{\bullet}(X, \k{G}_{\nu}), \tilde{\partial}^G)$, 
with filtration $F^{n} C_{\bullet}(X, \k{G}_{\nu}) = 
J^n \otimes_{\k} C_{\bullet}(X,\k)$, where 
$J=\bar{\nu}(I) =I_{\k} G$. 

\begin{lemma}
\label{lem:minfilt}
$\tilde{\partial}^G F^n C_{\bullet}(X, \k{G}_{\nu}) 
\subset F^{n+1} C_{\bullet}(X, \k{G}_{\nu})$, for all $n\ge 0$.
\end{lemma}

\begin{proof}
Recall from \S\ref{subsec:diff} that $p\circ \tilde{\partial} = 
\partial \circ p$, where 
$p\colon C_{\bullet}(\wX,\k) \surj C_{\bullet}(\wX,\k)/ 
F^1 C_{\bullet}(X,\k)$ is the quotient map. 
Consequently, $\tilde{\partial} F^0 \subset F^1$, 
by minimality of $X$. 

From the discussion in \S\ref{subsec:func}, we know that 
$(\nu\otimes \id)\circ \tilde{\partial} \tilde{\partial}^G\circ (\nu\otimes \id)$, 
and that $\nu \otimes \id$ preserves filtrations. 
Since $\nu \otimes \id$ is onto, we infer that 
$\tilde{\partial}^G F^0 \subset F^1$.  
The conclusion follows from 
$\k{G}$-linearity of $\tilde{\partial}^G$.
\end{proof}

\begin{definition}
\label{def:lin}
The {\em linearization} of the boundary map $\tilde{\partial}^G$
is the map induced by $\tilde{\partial}^G$ 
at the associated graded level,
\begin{equation*}
\label{eq:del lin}
\partial_G^{\linn} \colon \gr^*_F C_{\bullet} (X, \k{G}_{\nu})
\to  \gr^{*+1}_F C_{\bullet} (X, \k{G}_{\nu}).
\end{equation*}
\end{definition}

Now use again the minimality of $X$ to identify $C_q(X,\k) = H_q(X,\k)$ 
and $C_{\bullet} (X, \k{G}_{\nu})= \k{G} \otimes_{\k} H_q(X,\k)$.  
From the construction of the equivariant spectral sequence, 
we obtain immediately that
\begin{equation*}
\label{eq:freelin}
E^1( X, \k G_{\nu}) = E^0( X, \k G_{\nu}),\quad \text{and} 
\quad d^1_G= \partial_G^{\linn} .
\end{equation*}

At this point, it is easy to give a concrete interpretation 
of linearization, in terms of matrices and natural $\k$-bases provided by cells. 
For each $q\ge 1$,  
denote by $\Mat(\tilde{\partial}_q^G)$ the matrix 
corresponding to the boundary map 
\[
\tilde{\partial}_q^G \colon \k{G} \otimes_{\k} H_q(X,\k) 
\to \k{G} \otimes_{\k} H_{q-1}(X,\k).
\] 
By Lemma \ref{lem:minfilt}, all the entries of $\Mat(\tilde{\partial}_q^G)$ 
belong to the ideal $J=I_{\k}G$. Reducing those entries modulo 
$J^2$, we obtain a new matrix, denoted by 
$\Mat (\tilde{\partial}_q^G )\bmod J^2$, with entries in 
$J/J^2 \cong H_1(G, \k)$. 

\begin{corollary}
\label{cor:matlin}
Let $X$ be a minimal CW-complex, and let 
$\nu\colon \pi_1(X)\surj G$ be an epimorphism. Then, 
for all $q\ge 1$, 
\[
\Mat (\tilde{\partial}_q^G ) \bmod J^2= \Mat \big( \!
\xymatrixcolsep{50pt}
\xymatrix{H_q(X, \k) \ar^(.36){(\nu_{*}\otimes \id)\circ \nabla_X}[r] & 
H_1(G, \k)\otimes_{\k} H_{q-1}(X, \k)}\! \big) .
\]
\end{corollary}

\begin{proof}
With the identifications discussed above, $\Mat(\tilde{\partial}_q^G) 
\bmod J^2$ is the matrix of $\partial_G^{\linn}=d^1_G$. The 
conclusion follows from Proposition \ref{prop:d1 kg}.
\end{proof}

\begin{remark}
\label{rem:nuid}
The above equality was 
proved in \cite[Theorem 20]{DP1} for  $\nu =\id$ and $\k= \Z$, 
provided $H^*(X, \Z)$ is generated as a ring in degree one.%
\footnote{The sign in that Theorem comes from considering 
$C_{\bullet}(X, \Z \pi)$ as a right $\Z\pi$-module instead of a 
left  $\Z\pi$-module: see equations (10) and (13) from \cite{DP1}.} 
When $\k$ is a field, there is a connection between the 
chain complex $(E^1(X, \k\pi), d^1)$ and the well-known 
Koszul complex from homological algebra. We refer to 
\cite[Proposition 22]{DP1} for a class of Koszul resolutions 
coming from linearization, and to \cite[Theorem 23]{DP1} for 
an application to the computation of higher homotopy groups.
\end{remark}

\subsection{Linearizing the equivariant cochain complex}
\label{subsec:linearize cochains}
Let $X$ be a connected CW-complex, with $\pi=\pi_1(X)$, 
and let $\nu\colon \pi \surj G$ be an epimorphism. Recall 
from Example \ref{ex:covers} that 
\[
C^{\bullet}(X, {}_{\nu}\k{G})= 
( \Hom_{\k\pi}(C_{\bullet}(\wX,\k), {}_{\nu}\k{G}), \, 
\tilde{\delta}^{\bullet}_G)
\]
denotes the cochain complex of $X$, with coefficients in the 
left $\k\pi$-module ${}_{\nu}\k{G}$. This is a cochain complex 
of (right) $\k{G}$-modules, endowed with a decreasing filtration 
$F^{\bullet}$, with $n$-th term given by 
\[
F^n= \Hom_{\k\pi}(C_{\bullet}(\wX,\k), J^n),
\] 
with $J=I_{\k}G$ viewed as a left $\k\pi$-module via $\nu$. 
Alternatively, if we identify $C^q(X, {}_{\nu}\k{G})$ with 
$\Hom_{\k} (C_q(X,\k), \k{G})$, then  
$F^n$ corresponds to $\Hom_{\k} (C_q(X,\k), J^n)$. 

\begin{lemma}
\label{lem:mindfilt}
Suppose $X$ is a minimal CW-complex. Then 
$\tilde{\delta}_G F^n \subset F^{n+1}$, for all $n\ge 0$.
\end{lemma}

\begin{proof}
Follows immediately from the definition of $\tilde{\delta}_G$ and 
Lemma \ref{lem:minfilt} applied to $\nu= \id$.
\end{proof}

Thus, the coboundary map $\tilde{\delta}^{\bullet}_G$ 
induces a (dual) linearization map,
$\gr^*_F C^{\bullet} (X, {}_{\nu}\k{G})\to 
\gr^{*+1}_F C^{\bullet} (X,  {}_{\nu}\k{G})$, 
or, modulo standard identifications,
\begin{equation}
\label{eq:deltalin}
\delta_G^{\linn} \colon C^{\bullet}(X,\k)\otimes_{\k} \gr^*_J (\k{G}) 
\longrightarrow C^{\bullet}(X,\k)\otimes_{\k} \gr^{*+1}_J (\k{G}) .
\end{equation}
It is a routine matter to verify that the map $\delta_G^{\linn}$ 
above is $\gr_J (\k{G})$-linear.

Now fix pairs of natural dual bases for $H_q(X, \k)$ and $H^q(X, \k)$, 
for each $q\ge 1$. It is easy to check that
$\Mat(\tilde{\delta}_G^q)=\Mat(\tilde{\partial}^G_q)^{\tr}$, 
where $(\cdot)^{\tr}$ denotes the transpose.
By Corollary \ref{cor:matlin}, then, 
\begin{equation}
\label{eq:dmatlin}
\Mat(\tilde{\delta}_G^q)\bmod J^2= \Mat\!\big(\!%
\xymatrixcolsep{22pt}
\xymatrix{%
H_q(X, \k)\ar^(.38){(\nu_{*}\otimes \id)\circ \nabla_X}[rr] &&
H_1(G, \k)\otimes_{\k} H_{q-1}(X, \k)}\!\big)^{\tr}. 
\end{equation}

\subsection{Linearization and the Aomoto complex}
\label{subsec:lin aomoto}
Finally, we describe the dual linearization map in more familiar 
terms, in the case when $\nu\colon \pi \surj \pi_{\ab}$ 
is the abelianization homomorphism.  It turns out that this may 
be done in terms of the universal Aomoto complex, so
we begin by reviewing this notion.

Let $H^*$ be a graded $\k$-algebra. We need to assume 
that $H^1$ is a free, finitely-generated $\k$-module, and 
$a^2=0$, for all $a\in H^1$. (These conditions are satisfied by the 
cohomology rings of minimal CW-complexes, with arbitrary 
coefficients.) Pick a $\k$-basis $\{ e_1^*,\dots, e_n^* \}$ 
of $H^1$, and denote by $\{ e_1,\dots, e_n \}$ the 
dual basis of $H_1:= \Hom_{\k}(H^1, \k)$.  Denote by 
$S$ the symmetric algebra $\Sym(H_1)$, and identify 
it with the polynomial ring $\k [e_1,\dots, e_n]$. The 
{\em universal Aomoto complex} of $H$ is the 
cochain complex of free $S$-modules, 
\begin{equation*}
\label{eq:univaom}
\AA^{\bullet} (H)\colon 
\xymatrix{H^0 \otimes_{\k} S \ar^{D^0}[r] & 
H^1 \otimes_{\k} S \ar^{D^1}[r] & 
H^2 \otimes_{\k} S \ar^(.6){D^2}[r] & 
\cdots},
\end{equation*}
where the differentials are defined by 
\begin{equation}
\label{eq:aomoto diff}
D^{q-1}(\alpha \otimes 1)= \sum_{i=1}^{n} 
e_i^* \cdot \alpha \otimes e_i
\end{equation}
for $\alpha\in H^{q-1}$, and then extended by $S$-linearity. 
Our hypothesis on (strong) anti-commu\-tativity of $H^*$ 
in degree one easily implies that $D\circ D= 0$.

The terminology is motivated by the following universal 
property of $\AA^{\bullet}$.  Pick any element 
$z\in H^1= \Hom_{\k}(H_1, \k)$, and denote by 
$\ev_z \colon S\to \k$ the change of rings given by 
evaluation at $z$.  This leads to a specialization of 
$\AA^{\bullet}$, namely to the $\k$-cochain complex 
$\AA^{\bullet} (z):= \AA^{\bullet}\otimes_S \k$.
It is easy to check that $\AA^{\bullet} (z)$ coincides with 
the Aomoto complex of $H$ with respect to $z$, as 
defined in \S\ref{subsec:aomoto}.  We are now in position  
to state the last result of this paper.

\begin{theorem}
\label{thm:linaom}
Let $X$ be a minimal CW-complex. Then the linearization 
of the equivariant cochain complex of the universal abelian 
cover of $X$, with coefficients in $\k=\Z$ or a field,
coincides with the universal Aomoto complex of the 
cohomology ring $H^*(X,\k)$.
\end{theorem}

\begin{proof}
The dual linearized complex is described by 
\eqref{eq:deltalin} and \eqref{eq:dmatlin}. More precisely, 
for each $q\ge 1$, the transposed matrix of 
\[
\delta^{\linn}_{\pi_{\ab}} \colon H^{q-1}(X, \k)\otimes_{\k} S 
\to H^{q}(X, \k)\otimes_{\k} S ,
\]
coincides with the matrix of
\begin{equation}
\label{eq:matnabla}
\nabla_X \colon H_q(X, \k)\to H_1(X, \k)\otimes_{\k} H_{q-1}(X, \k)\, .
\end{equation}
On the other hand, $\nabla_X$ is the transpose of $\cup_X$.
Thus, $\delta^{\linn}_{\pi_{\ab}}$ coincides with the differential 
$D^{q-1}$ from \eqref{eq:aomoto diff}.  
\end{proof}

\begin{ack}
This work was started while the first author visited Northeastern 
University, in Spring, 2006. He thanks the Northeastern Mathematics 
Department for its support and hospitality during this visit. 
A substantial portion of the work was done while both authors 
visited the Abdus Salam International Centre for Theoretical
Physics in Trieste, Italy, in Fall, 2006.  We thank ICTP for its 
support and excellent facilities.
\end{ack}

\vspace{-2pc}
\newcommand{\arxiv}[1]
{\texttt{\href{http://arxiv.org/abs/#1}{arxiv:#1}}}
\renewcommand{\MR}[1]
{\href{http://www.ams.org/mathscinet-getitem?mr=#1}{MR#1}}
\newcommand{\doi}[1]
{\texttt{\href{http://dx.doi.org/#1}{doi:#1}}}

\end{document}